\documentclass[notitlepage,11pt,reqno]{amsart}
\usepackage[foot]{amsaddr}
\usepackage{amssymb,nicefrac,bm,upgreek,mathtools,verbatim}
\usepackage[final]{hyperref}
\usepackage[mathscr]{eucal}
\usepackage{dsfont}
\usepackage[normalem]{ulem}
\usepackage{amsopn}

\usepackage[margin=1in]{geometry}
\allowdisplaybreaks
\newcommand{\stkout}[1]{\ifmmode\text{\sout{\ensuremath{#1}}}\else\sout{#1}\fi}
\newtheorem{lemma}{Lemma}[section]
\newtheorem{theorem}{Theorem}[section]

\newtheorem{corollary}{Corollary}[section]

\theoremstyle{definition}
\newtheorem{definition}{Definition}[section]

\theoremstyle{remark}
\newtheorem{remark}{Remark}[section]
\numberwithin{theorem}{section}
\numberwithin{equation}{section}

\hypersetup{
  colorlinks=true,
  citecolor=mblue,
  linkcolor=mblue,
  urlcolor = blue,
  anchorcolor = blue,
  frenchlinks=false,
  pdfborder={0 0 0},
  naturalnames=false,
  hypertexnames=false,
  breaklinks}
\usepackage[capitalize,nameinlink]{cleveref}
\usepackage[abbrev,msc-links,nobysame,citation-order]{amsrefs} 

\crefname{section}{Section}{Sections}
\crefname{subsection}{Section}{Sections}
\crefname{condition}{Condition}{Conditions}
\crefname{hypothesis}{Hypothesis}{Conditions}
\crefname{assumption}{Assumption}{Assumptions}
\crefname{lemma}{Lemma}{Lemmas}
\crefname{fact}{Fact}{Facts}

\Crefname{figure}{Figure}{Figures}

\crefformat{equation}{\textup{#2(#1)#3}}
\crefrangeformat{equation}{\textup{#3(#1)#4--#5(#2)#6}}
\crefmultiformat{equation}{\textup{#2(#1)#3}}{ and \textup{#2(#1)#3}}
{, \textup{#2(#1)#3}}{, and \textup{#2(#1)#3}}
\crefrangemultiformat{equation}{\textup{#3(#1)#4--#5(#2)#6}}%
{ and \textup{#3(#1)#4--#5(#2)#6}}{, \textup{#3(#1)#4--#5(#2)#6}}%
{, and \textup{#3(#1)#4--#5(#2)#6}}

\Crefformat{equation}{#2Equation~\textup{(#1)}#3}
\Crefrangeformat{equation}{Equations~\textup{#3(#1)#4--#5(#2)#6}}
\Crefmultiformat{equation}{Equations~\textup{#2(#1)#3}}{ and \textup{#2(#1)#3}}
{, \textup{#2(#1)#3}}{, and \textup{#2(#1)#3}}
\Crefrangemultiformat{equation}{Equations~\textup{#3(#1)#4--#5(#2)#6}}%
{ and \textup{#3(#1)#4--#5(#2)#6}}{, \textup{#3(#1)#4--#5(#2)#6}}%
{, and \textup{#3(#1)#4--#5(#2)#6}}

\crefdefaultlabelformat{#2\textup{#1}#3}
\newcommand{\vertiii}[1]{{\left\vert\kern-0.25ex\left\vert\kern-0.25ex\left\vert #1 
    \right\vert\kern-0.25ex\right\vert\kern-0.25ex\right\vert}}

\newcommand{\Uadm}{\mathfrak U}
\newcommand{\Act}{\mathbb{U}}
\newcommand{\Usm}{\mathfrak U_{\mathsf{sm}}}

\newcommand{\Um}{\mathfrak U_{\mathsf{m}}}

\newcommand{\pV}{\mathrm{V}} 
\newcommand{\pv}{\mathrm{v}} 
\newcommand{\bS}{{\mathbb{S}}}
\newcommand{\cB}{{\sB}}  
\newcommand{\sB}{{\mathscr{B}}}  
\newcommand{\cC}{{\mathcal{C}}}   
\newcommand{\sE}{{\mathscr{E}}} 
\newcommand{\sF}{{\mathfrak{F}}}   
\newcommand{\cJ}{{\mathcal{J}}}  
  
\newcommand{\cK}{{\mathcal{K}}}  
\newcommand{\sL}{{\mathscr{L}}}  %
\newcommand{\Lp}{{L}}            

\newcommand{\cP}{{\mathcal{P}}}  
\newcommand{\sP}{{\mathscr{P}}}
\newcommand{\Lyap}{{\mathcal{V}}}  
\newcommand{\cX}{{\mathcal{X}}}

\newcommand{\RR}{\mathds{R}}
\newcommand{\NN}{\mathds{N}}

\newcommand{\Rd}{{\mathds{R}^{d}}}
\DeclareMathOperator{\Exp}{\mathbb{E}}

\newcommand{\D}{\mathrm{d}}

\newcommand{\cD}{\mathcal{D}} 
 
\newcommand{\Sob}{{\mathscr W}}    
\newcommand{\Sobl}{{\mathscr W}_{\text{loc}}} 

\newcommand{\df}{:=}
\newcommand{\transp}{^{\mathsf{T}}}

\DeclareMathOperator*{\trace}{Tr}

\DeclareMathOperator*{\diam}{diam}

\DeclareMathOperator*{\argmin}{arg\,min}

\newcommand{\order}{{\mathscr{O}}}
\newcommand{\sorder}{{\mathfrak{o}}}
\newcommand{\grad}{\nabla}
\newcommand{\uuptau}{{\Breve\uptau}}

\newcommand{\abs}[1]{\lvert#1\rvert}
\newcommand{\norm}[1]{\lVert#1\rVert}

\usepackage{color}
\definecolor{dmagenta}{rgb}{.4,.1,.5}
\definecolor{dblue}{rgb}{.0,.0,.5}
\definecolor{mblue}{rgb}{.0,.0,.7}
\definecolor{ddblue}{rgb}{.0,.0,.4}
\definecolor{dred}{rgb}{.7,.0,.0}
\definecolor{dgreen}{rgb}{.0,.5,.0}
\definecolor{Eeom}{rgb}{.0,.0,.5}

\begin{document}
\title[On the Approximation of Optimal Control in Regime-Switching Diffusions]
{On the Approximation of Optimal Control in Regime-Switching Diffusions}

\author[Somnath Pradhan]{Somnath Pradhan$^\dag$}
\address{$^\dag$Department of Mathematics, Indian Institute of Science Education and Research Bhopal,
Bhopal, MP - 462066, India}
\email{somnath@iiserb.ac.in}

\author[Dinesh Rathia]{Dinesh Rathia$^{\ddag}$}
\address{$^\ddag$Department of Mathematics, Indian Institute of Science Education and Research Bhopal,
Bhopal, MP - 462066, India}
\email{dinesh23@iiserb.ac.in}





\begin{abstract}
We study approximation and structural simplification of optimal control policies for controlled regime-switching diffusion processes for discounted, ergodic, finite-horizon, and exit-time criteria. We first establish continuity of the cost functionals over classes of Markov and stationary Markov policies by exploiting elliptic and parabolic regularity of the corresponding Hamilton--Jacobi--Bellman and Poisson equations. Using density results of policies with finite-action, piecewise-constant, and Lipschitz continuous, we show that each control problem admits $\varepsilon$-optimal policies within these structured subclasses. We then construct an Euler--Maruyama approximation of the controlled regime-switching diffusion under piecewise-constant controls. We prove strong convergence of the controlled state process and establish convergence of the associated finite-horizon value functions with rate $O(h^{\gamma/2})$. Building on this discretization, we develop a finite-state approximation of the induced discrete-time Markov chain via state-space quantization. We show that the value functions of the finite models converge uniformly on compact sets to the value function of the original problem, and that optimal policies of the approximating models are asymptotically optimal.

These results provide a systematic framework for approximating regime-switching diffusion control problems and justify the use of structured policies and finite-state models for numerical implementation.
\end{abstract}

\maketitle

\tableofcontents

\section{Introduction}

Stochastic optimal control is a fundamental area in applied mathematics and engineering, concerned with optimizing the behavior of dynamical systems evolving under uncertainty. Approximation and structural simplification of optimal control policies are therefore of central importance in both theory and applications. While the existence and characterization of optimal policies are well understood through dynamic programming and Hamilton–Jacobi–Bellman (HJB) equations \cites{FlemingSoner2006,Bor-book,KushnerDupuis2001,ABG-book,HP09-book,yong2012stochastic} and the Pontryagin maximum principle \cite{MR166037}, significantly less is known about whether optimal performance can be achieved using simpler or implementable classes of controls, or whether discrete-time approximations accurately capture the behavior of the underlying continuous-time system.

From both theoretical and computational perspectives, it is therefore essential to determine whether optimal policies can be approximated by structured strategies, such as finite-action, piecewise-constant, or Lipschitz continuous controls and whether discretized models provide accurate approximations for numerical computation, simulation, and learning-based methods. These questions are central to numerical algorithms, reinforcement learning, and practical implementation of stochastic control problems.

\medskip
\noindent
\textbf{Background and related work.}
Approximation of controlled diffusion processes has been studied extensively from several perspectives. A prominent approach is based on weak convergence and controlled Markov chain approximations, as developed by Kushner and Dupuis \cite{KushnerDupuis2001}, where discrete-time models are constructed to approximate both the continuous-time dynamics and the associated value functions. Another line of work relies on finite-difference schemes and viscosity solution methods for Hamilton--Jacobi--Bellman equations; see, for example, \cites{MR1466804,MR1759507,MR1668597,MR2177879,MR1916291}, where convergence and stability of numerical schemes are established. Classical numerical methods for stochastic differential equations, including the Euler--Maruyama scheme and its higher-order variants, are well understood; see \cite{KloedenPlaten1992}.

For controlled non-degenerate diffusions, Krylov \cite{MR1668597} showed that restricting controls to be constant over intervals of length $h^2$ yields an approximation of the value function with error of order $h^{1/3}$. Subsequent works, particularly \cites{MR2177879,MR1916291}, improved these estimates and obtained sharper convergence rates based on refinements of earlier techniques \cites{MR1466804,MR1759507,MR1668597}. More recently, Pradhan and Y\"uksel \cites{pradhan2025nearoptimalitydiscretetimeapproximations}*{Corollaries 4.8 and 5.4} established a convergence rate of order $h^{1/2}$ for controlled McKean--Vlasov diffusions over the space of piecewise constant controls.

Motivated by this, the present work establishes an explicit convergence rate of order $h^{\gamma/2}$, for $\gamma\in (0,\frac{1}{2})$, for controlled regime-switching diffusions under piecewise-constant controls.

From a structural viewpoint, continuity of induced costs under suitable control topologies plays a key role in approximation theory. Borkar \cite{Bor89} introduced a topology on the space of control policies that facilitates the analysis of continuity properties of performance criteria. Building on this framework, Pradhan and Y\"uksel \cites{SPSY2,SPSY} established the density of structured policy classes within the class of Markov and stationary Markov controls, leading to near-optimality results for controlled diffusions.

Approximation problems for discrete-time Markov decision processes (MDPs) have also been widely studied, including approximate dynamic programming, value and policy iteration, linear programming, reinforcement learning, state aggregation, etc; see \cites{MR2233994,MR3791811,Dufour04032015,1100984,133184,DUFOUR20121254,MR3036990,MR3722422,FOX1971665,doi:10.1287/moor.3.3.231,doi:10.1287/moor.4.2.179,doi:10.1287/moor.6.4.493}. In continuous time, discretization of diffusion models via schemes such as Euler–Maruyama (EM) has been extensively analyzed. For regime-switching diffusion processes (RSDPs)  with state-independent switching, strong and weak approximation results are available \cites{MAO2007936,YUAN2004223}. However, the state-dependent case is significantly more delicate due to the coupling between the continuous state and the switching mechanism. Even for uncontrolled models, obtaining sharp error bounds is challenging, with $L^1$, $L^2$, and $L^p$ convergence results appearing in \cites{ShaoEM,jin,math12121819}. In addition to time discretization, finite-state approximations play a central role in numerical stochastic control, as they enable the reduction of continuous-state problems to tractable Markov decision processes. Such approximations are well understood for Markov decision processes with both discrete and continuous state spaces \cites{Dufour04032015,MR3036990,DUFOUR20121254,1100984,133184,MR3722422}.

In this paper, we apply density results for structured control policies to the approximation of state-dependent \emph{controlled regime-switching diffusion processes} under piecewise-constant controls, with particular emphasis on the finite-horizon cost criterion.
 These hybrid systems, in which continuous stochastic dynamics interact with a controlled Markov chain, arise in applications such as finance, engineering systems, and queueing networks, but pose significant analytical challenges. While optimal control of RSDPs has been studied under finite-horizon \cite{FH}, discounted \cite{AGM93}, ergodic \cite{Ghosh97}, and exit-time criteria \cite{rathia}, a unified approximation framework combining structural policy simplification, time discretization and state approximation remains largely incomplete.

We introduce a two-step approximation framework for controlled regime-switching diffusions with state-dependent switching, integrating Euler–Maruyama discretization and finite-state quantization, and prove convergence of value functions together with asymptotic optimality of the induced policies. To the best of our knowledge, this is the first work that combines structural policy approximation, Euler–Maruyama discretization, and finite-state quantization for controlled regime-switching diffusions with state-dependent switching.

\medskip
\noindent
\textbf{Objective of the paper.}
The objective of this paper is to develop a systematic approximation framework for controlled RSDPs. Within a unified setting, we analyze discounted, ergodic, finite-horizon, and exit-time cost criteria, and investigate how optimal performance can be systematically approximated through structured control policies, then as an application of the near-optimality results for the finite horizon case, establish the near-optimality with respect to the time discretization and finite-state approximations.

\medskip
\noindent
\textbf{Main contributions.}
Our contributions are fourfold and establish a unified approximation framework combining structural policy simplification,  time discretization, and state-space quantization.
\begin{itemize}
\item \textbf{Continuity of cost functionals.}
We establish continuity of the induced cost functionals with respect to Markov and stationary Markov policies under the Borkar topology. The analysis relies on regularity properties of the associated HJB and Poisson equations. These results play an important role in approximation and near-optimality.

\item \textbf{Near-optimality of structured policies.}
Using known density results for finite-action, piecewise-constant, and Lipschitz continuous stationary Markov policies, together with the continuity results established in this paper, we show that each control problem admits $\varepsilon$-optimal policies within these structured subclasses.
\end{itemize}
Moreover, as an application of the above near-optimality results we obtain following Markov chain approximation results for the finite-horizon case:
\begin{itemize}
\item \textbf{Discrete-time approximation.}
We construct an Euler–Maruyama Markov chain approximation under piecewise-constant controls and analyze the resulting controlled discrete-time model. We establish strong convergence of the controlled state process and prove convergence of discrete-time value functions to their continuous-time counterparts, with an explicit error bound of order $O(h^{\gamma/2})$, \(\gamma\in (0,\frac{1}{2})\). Consequently, for any $\varepsilon > 0$, optimal policies of the
time-discretized model yield $\varepsilon$-optimal performance for the original system.

\item \textbf{Finite-state approximation and asymptotic optimality.}
Building on the discrete-time approximation, we develop a finite-state approximation via state-space quantization of the induced Markov chain. We show that the value functions of the finite-state models converge uniformly on compact sets to those of the original model and that optimal policies
obtained from these models are asymptotically optimal for the original control problem.
\end{itemize}
Taken together, these results provide a unified and rigorous framework for approximating regime-switching stochastic control problems and justify the use of structured policies and discretized models in numerical implementations.

\medskip
\noindent
\textbf{Organization of the paper.}
\cref{PD} introduces the model and assumptions. \cref{Cost} formulates the cost criteria and policy spaces. Continuity results are developed in \cref{continuity}. \cref{Denseness}--\cref{Section-near-op} leverage existing density results to establish the near optimality of structured policies. \cref{Application} develops the Euler--Maruyama approximation and convergence of discrete-time value functions, followed by the finite-state approximation and asymptotic optimality results.
\section{ Description of the problem}\label{PD} Let $\Act$ be a compact metric space of control actions, and let $\pV=\mathscr{P}(\Act)$ denote the space of probability measures on $\Act$, 
equipped with the topology of weak convergence.
Consider the controlled  RSDP $(X_t, S_t)$ taking values in $\Rd \times \mathbb{S}$, where $\mathbb{S} = \{1, \dots, N\}$ is a finite set of regimes. The process is defined on a complete probability space $(\Omega,\sF,\mathbb{P})$, and its dynamics are governed by the following stochastic differential equations:
\begin{equation}\label{E1.1}
\begin{aligned}
dX_t &= b(X_t, S_t, U_t)\,dt + \upsigma(X_t, S_t)\,dW_t,\\
dS_t &= \int_{\RR} h(X_t, S_{t-}, U_t, z)\, \mathcal{P}(dt, dz)
\end{aligned}
\end{equation}
where 
\begin{itemize}
\item $X_0 $ and $S_0$ denote the prescribed initial distribution of the diffusion and regime processes, respectively.

\item The functions, $b=[b_1,\ldots,b_d]^{\transp} : \Rd \times \mathbb{S} \times \Act \to \Rd$ is the drift coefficient, and
$\upsigma=[\upsigma_{ij}]_{1\le i,j\le d} : \Rd \times \mathbb{S} \to \RR^{d\times d}$ is the diffusion matrix.

    \item $W$ is a $d$-dimensional standard Wiener process.
    
    \item $\mathcal{P}(dt, dz)$ is a Poisson random measure on $\RR_+ \times \RR$ with intensity $dt \times {\bf m}(dz)$, where ${\bf m}$ is the Lebesgue measure on $\RR$.
    
    \item $\mathcal{P}(\cdot, \cdot)$, $W(\cdot)$, $X_0$, $S_0$ are independent.
    
    \item The jump function $h : \Rd \times \mathbb{S} \times \Act \times \RR \to \RR$ is defined by
     \begin{equation}\label{markov-h}
        h(x, i, \zeta, z) := 
        \begin{cases}
            j - i & \text{if } z \in \Delta_{ij}(x, \zeta), \\
            0 & \text{otherwise},
        \end{cases}
     \end{equation}

    where for each  $(x,\zeta)$ and $i,j\in\mathbb{S}$, the sets $\Delta_{ij}(x, \zeta)$ are left-closed, right-open disjoint intervals of $\RR$ having length $m_{ij}(x, \zeta)$ that partition $\RR$.
    For each $i\neq j\in\mathbb{S}$, define
\[
g_{ij}
:=
\sup_{(x,\zeta)\in\Rd\times\Act} m_{ij}(x,\zeta)<\infty.
\]
We define the intervals $\Delta_{ij}(x,\zeta)$, $i\neq j$, as follows:
\begin{align*}
\Delta_{12}(x,\zeta)
&:=[0,m_{12}(x,\zeta)), \nonumber\\
\Delta_{21}(x,\zeta)
&:=[g_{12},\, g_{12}+m_{21}(x,\zeta)), \nonumber\\
\Delta_{13}(x,\zeta)
&:=[g_{12}+g_{21},\, g_{12}+g_{21}+m_{13}(x,\zeta)), \nonumber\\
\Delta_{31}(x,\zeta)
&:=[g_{12}+g_{21}+g_{13},\, 
g_{12}+g_{21}+g_{13}+m_{31}(x,\zeta)), \nonumber\\
\Delta_{23}(x,\zeta)
&:=[g_{12}+g_{21}+g_{13}+g_{31},\,
g_{12}+g_{21}+g_{13}+g_{31}+m_{23}(x,\zeta)), \nonumber\\
\Delta_{14}(x,\zeta)
&:=[g_{12}+g_{21}+g_{13}+g_{31}+g_{23}, \nonumber\\
&\qquad
g_{12}+g_{21}+g_{13}+g_{31}+g_{23}+m_{14}(x,\zeta)), \nonumber\\
\Delta_{41}(x,\zeta)
&:=[g_{12}+g_{21}+g_{13}+g_{31}+g_{23}+g_{14},
\nonumber\\
&\qquad
g_{12}+g_{21}+g_{13}+g_{31}+g_{23}+g_{14}+m_{41}(x,\zeta)),
\nonumber\\
&\qquad\vdots
\end{align*}

Note that for any $(i,j)\neq (k,l)$ and $x\in\Rd$,
\[
\Delta_{ij}(x,\zeta)\cap \Delta_{kl}(x,\zeta)=\emptyset.
\]
Moreover, for any $i\neq j\in \mathbb{S}$ and $x,\bar x\in\Rd$,
the intervals $\Delta_{ij}(x,\zeta)$ and $\Delta_{ij}(\bar x, \bar \zeta)$ have the same left endpoint.
 For convenience, we set $\Delta_{ii}(x,\zeta)=\emptyset$ and $\Delta_{ij}(x,\zeta)=\emptyset$ if $m_{ij}(x,\zeta)=0$.
    
    \item $\mathbb{M} := (m_{ij})_{i,j \in \mathbb{S}}$ denotes the transition-rate matrix of the controlled Markov chain $S_t$, where $m_{ij}:\Rd \times \Act\to \RR$ are the switching rates such that $m_{ij} \geq 0$ for $i \neq j$ and, $\sum_{j=1}^{N} m_{ij} = 0$, $i \in \mathbb{S}$.

    We assume switching rates are bounded (i.e., there exists   $M>0\;\text{ such that }\;\norm{m_{ij}}_{\infty}\le M,\;\forall\;i,j\in\mathbb{S} \;$)
    throughout this article. 
    
    \item The control process $\{U_t\}$ takes values in $\pV$, is progressively measurable with respect to $\sF_t := \text{completion of } \sigma\{X_s, S_s; s \leq t\}$ relative to $(\sF, \mathbb{P})$, and is non-anticipative: for each $t \geq 0$, the $\sigma$-field $\sigma\{U_s\,;\, s \leq t\}$ is independent of 
    \[
        \sigma\{W_s - W_t,\, \cP(A, B) : A \in \mathcal{B}([s, \infty)), B \in \mathcal{B} (\RR), s \geq t\}.
    \]
    The process $U$ is called an \textit{admissible control}, and the set of all admissible controls is denoted by $\Uadm$ (see, \cite{ABG-book}*{Chapter 5, p. 197}).
\item 
 For relaxed controls $\mathrm v \in \pV=\mathscr P(\Act)$, the drift \(b: \Rd \times \mathbb{S} \times \pV \to \Rd\) is extended by
\begin{equation*}
b (x,i,\mathrm{v}) = \int_{\Act} b(x,i,\zeta)\mathrm{v}(\D \zeta), 
\end{equation*}
\end{itemize}

To ensure existence and uniqueness of strong solutions to \cref{E1.1}, we impose the following structural assumptions on the drift coefficient $b$, the diffusion matrix $\upsigma$, and the transition rate matrix $\mathbb{M}$.
 
\subsection{Assumptions}\hspace{50em} 
\vspace{2mm}

Throughout the paper we impose the following structural conditions on the coefficients of \cref{E1.1}.
\vspace{0.01mm}
\begin{itemize}
\item[\hypertarget{A1}{(A1)}]
\emph{Local Lipschitz continuity:\/}
The functions $b(x,i,\zeta),\upsigma^{ij}(x,k),\,m_{ij}(x,\zeta)$,\,\,
are continuous and locally Lipschitz continuous in $x$ (uniformly with respect $\zeta$) with a Lipschitz constant $C_{R}>0$
depending on $R>0$, i.e.,
\begin{equation*}
\abs{b(x,i,\zeta) - b(y,i, \zeta)}^2 + \norm{\upsigma(x,i) - \upsigma(y,i)}^2\,+\,|m_{ij}(x,\zeta)-m_{ij}(y,\zeta)|^2 \,\le\, C_{R}\,\abs{x-y}^2
\end{equation*}
for all $x,y\in \sB_R\,,\,i,j \in \mathbb{S}$ and $\zeta\in\Act$, where $\norm{\upsigma}\df\sqrt{\trace(\upsigma\upsigma\transp)}$\,. 

\medskip
\item[\hypertarget{A2}{(A2)}]
\emph{Affine growth condition:\/} The drift term
$b$ and the diffusion coefficient $\upsigma$ satisfy a global growth condition of the form
\begin{equation*}
\sup_{\zeta\in\Act}\, \langle b(x,i, \zeta),x\rangle^{+} + \norm{\upsigma(x,i)}^{2} \,\le\,C_0 \bigl(1 + \abs{x}^{2}\bigr) 
\end{equation*}
for all $x\in\Rd,\,i \in \mathbb{S}$ and for some constant $C_0>0$.

\medskip
\item[\hypertarget{A3}{(A3)}]
\emph{Nondegeneracy:\/}
For each $R>0$, it holds that
\begin{equation*}
\sum_{i,j=1}^{d} a^{ij}(x,k)z_{i}z_{j}
\,\ge\,C^{-1}_{R} \abs{z}^{2} \qquad\forall\, x\in \sB_{R}\,, k\in \mathbb{S},
\end{equation*}
and for all $z=(z_{1},\dotsc,z_{d})\transp\in\Rd$,
where $a\df \frac{1}{2}\upsigma \upsigma\transp$.
\end{itemize}
\vspace{1em}
 Under these assumptions \hyperlink{A1}{(A1)}--\hyperlink{A3}{(A3)}, the system \cref{E1.1}
admits a unique, strong solution  for every admissible control (see, for example, \cite{ABG-book}*{p. 197} and \cite{Yin10}*{Theorem 3.10}), with 
\[
X \in \cC(\RR_+; \Rd), \quad S \in\cD(\RR_+; \mathbb{S}),
\]
where \( \cD(\RR_+; \mathbb{S}) \) is the space of all right-continuous functions from \( \RR_+ \) to \( \mathbb{S} \) having left limits.

The ergodic behavior of the joint process $Y_t := (X_t, S_t)$ depends strongly on the coupling
coefficients $\{m_{ij}\}$. For this, we define the matrix
\[
\widetilde{\mathbb{M}}(x, \zeta) := (\widetilde{m}_{ij}(x, \zeta)) : \Rd \times  \Act \to \RR^{N \times N},
\]
where
\[
\widetilde{m}_{ij}(x, \zeta) := 
\begin{cases}
m_{ij}(x, \zeta), & \text{if } i \ne j, \\
0, & \text{otherwise}.
\end{cases}
\]

In addition to the usual structural assumptions \hyperlink{A1}{(A1)}--\hyperlink{A3}{(A3)}, we impose the following condition:

\vspace{1em}
\noindent
\begin{itemize}
\item[\hypertarget{A4}{(A4)}]
\emph{Irreducibility:\/}
The matrix \( \breve{\mathbb{M}}(x) := (\breve{m}_{ij}(x)) \), where 
\[
\breve{m}_{ij}(x) := \min_{\zeta \in  \Act} \widetilde{m}_{ij}(x, \zeta),
\]
is irreducible in \( \Rd \), that is, for every nonempty disjoint sets  \( \,\mathbb{S}_1, \mathbb{S}_2 \subset \mathbb{S} \) satisfying \( \mathbb{S}_1 \cup\, \mathbb{S}_2 = \mathbb{S} \), there exist \( i_0 \in \mathbb{S}_1 \) and \( j_0 \in \mathbb{S}_2 \) such that
\[
\big|\{ x \in \Rd : \breve{m}_{i_0 j_0}(x) > 0 \} \big| > 0,
\]
where \( | \cdot | \) denotes the Lebesgue measure. 
\end{itemize}
\vspace{2mm}
In this article, we consider the problem of minimizing discounted, finite horizon, exit-time, and ergodic cost criteria.
\vspace{1em}
Let $c\colon\Rd\times \mathbb{S} \times \Act \to \RR_+$ be the \emph{running cost} function. We assume that 
\begin{itemize}
\item[\hypertarget{A5}{{(A5)}}]
The \emph{running cost} $c$ is bounded (i.e., there exist $M>0$ such that $\|c\|_{\infty} \leq M$), continuous and locally Lipschitz continuous in $x$ uniformly with respect to $\zeta\in\Act$.
\end{itemize}
 For relaxed controls $\mathrm v \in \pV=\mathscr P(\Act)$, the running cost $c\colon\Rd\times \mathbb{S}\times\pV \to\RR_+$ is extended by 
\begin{equation*}
c(x,i,\pv) := \int_{\Act}c(x,i,\zeta)\pv(\D\zeta)\,.
\end{equation*}

\subsection{Cost criteria:}\label{Cost} 

The following cost evaluation criteria will be considered in this article.
\vspace{2mm}
\subsubsection{\bf Discounted cost criterion.} \hspace{50em}
\vspace{0.5em}

For any admissible control $U \in\Uadm$, the associated \emph{$\alpha$-discounted cost} is defined by
\begin{equation}\label{EDiscost}
\cJ_{\alpha}^{U}(x,i, c) \,\df\, \Exp_{x,i}^{U} \left[\int_0^{\infty} e^{-\alpha t} c(X_s,S_s, U_s) \D s\right],\quad (x,i)\in\Rd \times \mathbb{S}\,,
\end{equation} where $\alpha > 0$ is the discount factor, $(X_{(\cdot)},S_{(\cdot)})$ is the solution of the controlled system  \cref{E1.1} under $U \in\Uadm$, and $\Exp_{x,i}^{U}$ denotes the expectation with respect to the law of the process $(X_{(\cdot)},S_{(\cdot)})$ with the initial condition $(x,i)$. The control objective is to minimize the cost in~\eqref{EDiscost} over all admissible controls. A control $U^{*}\in\Uadm$ is said to be \emph{optimal} if, for every $(x,i)\in\Rd\times\mathbb{S}$,
\begin{equation}\label{OPDcost}
\cJ_{\alpha}^{U^*}(x,i, c) = \inf_{U\in \Uadm}\cJ_{\alpha}^{U}(x,i, c) \,\,\, (\,=:\, \,\, V_{\alpha}(x,i))\,,
\end{equation} where $V_{\alpha}(x,i)$ is called the $\alpha$-discounted optimal value function.

\vspace{2mm}
\subsubsection{\bf Ergodic cost criterion.}\hspace{50em}
\vspace{0.5em}

For a control $U \in \Uadm$, the corresponding \emph{ergodic cost functional} is defined as
\begin{equation*}\label{ECost1}
\sE_{x,i}(c,U) = \limsup_{T\to \infty}\frac{1}{T}\Exp_{x,i}^{U}\left[\int_0^{T} c(X_s,S_s, U_s) \D{s}\right]\,,\qquad (x,i)\in\Rd\times\mathbb{S}
\end{equation*} and the optimal value is defined as
\begin{equation*}\label{ECost1Opt}
\sE^*(c) \,\df\, \inf_{(x,i)\in\Rd\times \mathbb{S}}\,\inf_{U\in \Uadm}\sE_{x,i}(c, U)\,.
\end{equation*}
Then a control $U^*\in \Uadm$ is said to be optimal if we have 
\begin{equation*}\label{ECost1Opt1}
\sE_{x,i}(c, U^*) = \sE^*(c)\,.
\end{equation*}

\subsubsection{\bf Finite horizon cost.}\hspace{50em}
\vspace{0.5em}

For any $U\in \Uadm$, the associated \emph{finite horizon cost} is given by
\begin{equation*}\label{FiniteCost1}
\cJ_{T}^U(x,i,c) = \Exp_{x,i}^{U}\left[\int_0^{T} c(X_s,S_s,U_s) \D{s} + c_{_{T}}(X_T,S_T)\right]\,,
\end{equation*} where $c_{_{T}}(\cdot,\cdot)$ is the terminal cost. The optimal value is defined as
\begin{equation*}\label{FiniteCost1Opt}
\cJ_{T}^*(x,i,c) \,\df\, \inf_{U\in \Uadm}\cJ_{T}^U(x,i,c)\,.
\end{equation*}
Thus, a policy $U^*\in \Uadm$ is said to be (finite horizon) optimal if we have 
\begin{equation*}\label{FiniteCost1Opt1}
\cJ_{T}^{U^*}(x,i,c) = \cJ_{T}^*(x,i,c)\quad \text{for all}\,\,\, (x,i)\in \Rd\times\mathbb{S},.
\end{equation*}
We assume the terminal function $c_{_{T}}$ satisfies $c_{_{T}} \in \Sobl^{2,p}(\Rd\times \mathbb{S}) \cap L^\infty(\Rd\times \mathbb{S})$ for some $p \ge 2$ throughout this paper.
\vspace{2mm}

\subsubsection{\bf Cost up to an exit time.}\hspace{50em}
\vspace{0.5em}

For each $U\in\Uadm$, the associated exit time cost is defined as
 \[
    \hat{\cJ}^U_e(x,i) := \Exp_{x,i}^U\!\left[
        \int_0^{\uptau(\mathcal{O})} e^{-\int_0^t \beta(X_s,S_s,U_s)\,ds}\, c(X_t,S_t,U_t)\,dt
        + e^{-\int_0^{\uptau(\mathcal{O})} \beta(X_s,S_s,U_s)\,ds}\, h(X_{\uptau(\mathcal{O})},S_{\uptau(\mathcal{O})})
    \right]
   \]
where $\mathcal{O} \subset \Rd$ is a bounded domain, $\beta(\cdot,\cdot,\cdot): \bar{\mathcal{O}}\times\mathbb{S} \times \Act \to [0,\infty)$ is the discount function, and 
$h:\bar{\mathcal{O}}\times\mathbb{S}\to \RR_+$ is the terminal cost function. The optimal value is defined as
    \begin{equation*}
         \hat{\cJ}^{*}_e(x,i)=\inf_{U\in \Uadm}\hat{\cJ}^{U}_e(x,i)
    \end{equation*}
    and a control $U^*\in \Uadm$ is said to be optimal if we have 
\begin{equation*}
     \hat{\cJ}^{U^*}_e(x,i)=\hat{\cJ}^{*}_e(x,i)=\inf_{U\in \Uadm}\hat{\cJ}^{U}_e(x,i)
\end{equation*}

\subsection{A topology on Control Policies}\label{top-sec}
\begin{definition}(\textit{Markov control}:)
    An admissible control is called a \emph{Markov control} if it is of the form \(U_t\,=\, v(t, X_t, S_t) \), for some Borel measurable function $v : \RR_+ \times \Rd \times \mathbb{S} \to  \pV$.\\ We  denote $\Um$ as the space of all Markov controls.
\end{definition}
\begin{definition}
(\emph{Stationary and Stable Stationary Markov Controls:})
 If the function $v$ in the above definition is independent of $t$, then $U$, or by an abuse of notation $v$ itself, is called a \emph{stationary Markov control}.
 We denote the set of all such controls by \( \Usm \).\\
A stationary Markov control \(v \in \Uadm_{\mathrm{sm}}\) is said to be \emph{stable} if the 
corresponding controlled RSDP \((X_t, S_t)\) is positive recurrent.
The set of all stable stationary Markov controls is denoted by \( \Uadm_{\mathrm{ssm}} \subset \Usm \).
\end{definition}

The hypotheses in \hyperlink{A1}{(A1)}--\hyperlink{A3}{(A3)} also imply the existence of unique strong solutions under Markov controls, which is a strong Feller (therefore strong Markov) process (see \cite{AGM93}*{Theorem 2.1} and \cite{ABG-book}*{Theorem 5.2.9}).
 From \cite{AGM93}*{Section 3}, we have that the set $\Usm$ is metrizable with compact metric with the following topology: A sequence $v_n\to v$ in $\Usm$ if and only if
\begin{equation*}
\lim_{n\to\infty}\int_{\Rd}f(x,i)\int_{\Act}g(x,i,\cdot)v_{n}(x,i)(\D \zeta)\D x = \int_{\Rd}f(x,i)\int_{\Act}g(x,i,\cdot)v(x,i)(\D \zeta)\D x
\end{equation*}
for all $f\in \Lp^1(\Rd\times \mathbb{S})\cap \Lp^2(\Rd\times \mathbb{S})$, $g\in \cC_b(\Rd\times \mathbb{S} \times  \Act)$ and $i\in \mathbb{S}$ (for more details, see \cite{AGM93}*{Lemma~3.2})\,.
Similarly, in view of \cite{Bor89}, from \cite{SPSY}*{Definition 2.2} we say a sequence $v_n \to v$ in $\mathcal{\Um}$ if and only if
\begin{align*}
\lim_{n \to \infty} 
&\int_0^\infty \int_{\Rd} f(t,x,i) 
\Bigg( \int_{\Act} g(t,x,i,\zeta)\, v_n(t,x,i)(d\zeta) \Bigg) dx\,dt\nonumber\\
&=
\int_0^\infty \int_{\Rd} f(t,x,i) 
\Bigg( \int_{\Act} g(t,x,i,\zeta)\, v(t,x,i)(d\zeta) \Bigg) dx\,dt,
\end{align*}
for all 
$i\in\mathbb{S,}\,\, f \in L^1([0,\infty)\times\Rd \times\mathbb{S}) \cap L^2([0,\infty)\times\Rd \times\mathbb{S} ),\,\, g \in \cC_b([0,\infty)\times\Rd\times\mathbb{S} \times \Act).$

\vspace{0.5em}
We define a family of operators $\sL_{\zeta}$ mapping
$\cC^2(\Rd\times \mathbb{S})$ to $\cC(\Rd\times \mathbb{S})$ by
\begin{equation*}\label{E-cI}
\sL_{\zeta} f(x,i) \,\df\, \trace\bigl(a(x,i)\grad^2 f(x,i)\bigr) + \,b(x,i,\zeta)\cdot \grad f(x,i)\,+\, \sum_{j \in \mathbb{S}} m_{ij}(x,\zeta) f(x,j)\,, 
\end{equation*}
for $\zeta\in\Act$, $f\in \cC^2(\Rd\times \mathbb{S})$\,.
For $\pv \in\pV$ we extend $\sL_{\zeta}$ as follows:
\begin{equation*}\label{EExI}
\sL_\pv f(x,i) \,\df\, \int_{\Act} \sL_{\zeta} f(x,i)\pv(\D \zeta)\,.
\end{equation*} For $v \in\Usm$, we define
\begin{equation*}\label{Efixstra}
\sL_{v} f(x,i) \,\df\, \trace(a(x,i)\grad^2 f(x,i)) + b(x,i,v(x,i))\cdot\grad f(x,i)\,+\, \sum_{j \in \mathbb{S}} m_{ij}(x,v(x,i)) f(x,j)\, .
\end{equation*}

\subsection{Problem Studied}\label{CRPB}
In this work, our primary objective is to address the following fundamental questions:
\begin{itemize}
    \item \textbf{Continuity of finite and infinite horizon costs.}  
    Suppose $v_n\in \Usm$ is a sequence of Markov controls such that $v_n\to v\in \Usm$ in topology defined in \cref{top-sec} (Borkar topology). 
    Does this imply convergence of the corresponding cost functions, namely:
    \begin{itemize}
        \item[•] \emph{Discounted cost:} $\cJ_{\alpha}^{v_n}(x,i,c) \;\to\; \cJ_{\alpha}^{v}(x,i,c)$ ? 
        \item[•] \emph{Ergodic cost:} $\mathcal{E}_{x,i}(c, v_n) \;\to\; \mathcal{E}_{x,i}(c,v)$ ? 
        \item[•] \emph{Finite-horizon cost:} $\mathcal{J}_{T}^{v_n}(x,i,c) \;\to\; \mathcal{J}_{T}^v(x,i,c)$ ? 
        \item[•] \emph{Exit-time cost:} $\hat{\mathcal{J}}_{e}^{v_n}(x,i,c) \;\to\; \hat{\mathcal{J}}_{e}^v(x,i)$ ?
    \end{itemize}

    \item \textbf{Near Optimality of smooth and quantized policies.}  
    For given $\varepsilon>0$, does there exist a policy $v_\varepsilon$ which is smooth (Lipschitz) or quantized (finite action/ piecewise constant)  such that it is near optimal? i.e,
    \begin{itemize}
        \item[•] \emph{Discounted cost:} $\cJ_{\alpha}^{v_\varepsilon}(x,i,c) \leq \cJ_{\alpha}^{v}(x,i,c)+\varepsilon$ ? 
        \item[•] \emph{Ergodic cost:} $\mathcal{E}_{x,i}(c, v_\varepsilon) \leq \mathcal{E}_{x,i}(c,v)+\varepsilon$ ? 
        \item[•] \emph{Finite-horizon cost:} $\mathcal{J}_{T}^{v_\varepsilon}(x,i,c) \leq \mathcal{J}_{T}^v(x,i,c)+\varepsilon$ ? 
        \item[•] \emph{Exit-time cost:} $\hat{\mathcal{J}}_{e}^{v_\varepsilon}(x,i,c) \leq \hat{\mathcal{J}}_{e}^v(x,i)+\varepsilon$ ?
    \end{itemize}
    \item \textbf{Approximation of optimal policies.}  
If the original system is discretized (in time or state), do the optimal policies 
obtained from the discretized models yield vanishing performance loss for the true system? 
In particular, for the finite-horizon cost, if $v^h$ and $v^n$ denote the optimal policies 
for the time- and state-discretized models, respectively, do we have
\begin{itemize}
    \item[•] Time discretization:
    \(
    \cJ_T^{v^h}(x,i) \to \cJ_T^*(x,i)
    \quad \text{as } h \to 0\, ?
    \)
    
    \item[•] State discretization:
    \(
    \cJ_T^{v^n}(x,i) \to \cJ_T^*(x,i)
    \quad \text{as } n \to \infty \, ?
    \)
\end{itemize}
\end{itemize}
Now we introduce the notations that will be used throughout the rest of the article.
\subsection*{Notation:}
\begin{itemize}
\item For any set $A\subset\Rd$, by $\uptau(A)$ we denote \emph{first exit time} of the process $(X_{t},S_t)$ from the set $A\subset\Rd$, defined by
\begin{equation*}
\uptau(A) \,\df\, \inf\,\{t>0\,\colon (X_{t},S_t)\not\in A\times \mathbb{S}\}\,.
\end{equation*}
\item $\sB_{r}$ denotes the open ball of radius $r$ in $\Rd$, centered at the origin, and $\sB_{r}^c$ denotes the complement of $\sB_{r}$ in $\Rd$\,.
\item $\uptau_{r}$, $\uuptau_{r}$ denote the first exit time from $\sB_{r}$, $\sB_{r}^c$ respectively, i.e., $\uptau_{r}\df \uptau(\sB_{r})$, and $\uuptau_{r}\df \uptau(\sB^{c}_{r})$.
\item By $\trace A$ we denote the trace of a square matrix $A$.
\item For any domain $\cD\subset\Rd$, the space $\cC^{k}(\cD)$ ($\cC^{\infty}(\cD)$), $k\ge 0$, denotes the class of all real-valued functions on $\cD$ whose partial derivatives up to and including order $k$ (of any order) exist and are continuous.
\item $\cC_{\mathrm{c}}^k(\cD)$ denotes the subset of $\cC^{k}(\cD)$, $0\le k\le \infty$, consisting of functions that have compact support. This denotes the space of test functions.
\item $\cC_{b}(\Rd)$ denotes the class of bounded continuous functions on $\Rd$\,.
\item $\cC^{k}_{0}(\cD)$ denotes the subspace of $\cC^{k}(\cD)$, $0\le k < \infty$, consisting of functions that vanish in $\cD^c$.
\item $\cC^{k,r}(\cD)$ denotes the class of functions whose partial derivatives up to order $k$ are H\"older continuous of order $r$.
\item $\Lp^{p}(\cD)$, $p\in[1,\infty)$ denotes the Banach space
of (equivalence classes of) measurable functions $f$ satisfying
$\int_{\cD} \abs{f(x)}^{p}\,\D{x}<\infty$.
\item $\Sob^{k,p}(\cD)$, $k\ge0$, $p\ge1$ denotes the standard Sobolev space of functions on $\cD$ whose weak derivatives up to order $k$ are in $\Lp^{p}(\cD)$, equipped with its natural norm (see, \cite{Adams})\,.
\item  If $\mathcal{X}(Q)$ is a space of real-valued functions on $Q$, $\mathcal{X}_{\mathrm{loc}}(Q)$ consists of all functions $f$ such that $f\varphi\in\mathcal{X}(Q)$ for every $\varphi\in\cC_{\mathrm{c}}^{\infty}(Q)$. In a similar fashion, we define $\Sobl^{k, p}(\cD)$. 


\item We also adopt the notation \( \mathcal{X}(Q \times\mathbb{S}) \) to indicate the product space \( (\mathcal{X}(Q))^N \), where \( N \) is the cardinality of \( \mathbb{S} \). The corresponding norm on \( \mathcal{X}(Q \times\mathbb{S}) \) is defined by
\[
\|f\|_{\mathcal{X}(Q \times\mathbb{S})} := \sum_{k \in\mathbb{S}} \|f_k\|_{\mathcal{X}(Q)}
\]

Let \( f \in \cC(\Rd \times\mathbb{S}) \), then, by \( f \gg 0 \), we mean that \( f_k > 0 \) for all \( k \in\mathbb{S} \).

\item If $h \in \cC(\Rd \times \mathbb{S} \times  \Act)$ with $h>0$, $\sorder(h)$ denotes the set of functions $f \in \cC(\Rd \times \mathbb{S} \times  \Act)$ having the property 
\begin{equation*}
\limsup_{|x| \to \infty} \sup_{\zeta\in  \Act} \sup_{i \in \mathbb{S}} \frac{|f(x, i, \zeta)|}{h(x, i, \zeta)} = 0
\end{equation*}


\end{itemize}

\section{ Continuity of the cost functions}\label{continuity}
In this section, we show that the cost functions introduced in \cref{Cost} are continuous with respect to the control policies.
\subsection{Continuity of discounted cost}\hspace{50em}
\vspace{0.5em}

The following theorem proves the continuity of the $\alpha$-discounted cost with respect to the control policies.
\begin{theorem}\label{TH2.1}
Suppose Assumptions \hyperlink{A1}{(A1)}–\hyperlink{A3}{(A3)} and \hyperlink{A5}{(A5)} hold. Then the map $v \mapsto \cJ_{\alpha}^{v}(x,i,c)$ from $\Usm$ to $\RR$ is continuous.
\end{theorem}

\begin{proof}
Let $v_n$ be a sequence in $\Usm$ such that $v_n\to v$ in $\Usm$. From \cite{AS}*{Theorem 2.1} for each $n\in \NN$ there exists a unique solution \(\cJ^{v_n}_{\alpha}\in \Sobl^{2,p}(\Rd\times\mathbb{S})\) to the following Poisson equation
\begin{equation}\label{ETC1.4B}
\sL_{v_n}\varphi(x,i,c)+c(x,i,v_n(x,i))=\alpha\varphi(x,i,c)
 \end{equation}
Rewriting the above equation, we have 
\begin{align*}
&\trace\bigl(a(x,i)\grad^2 \cJ^{v_n}_\alpha(x,i,c)\bigr) + b(x,i,v_n(x,i))\cdot \grad \cJ^{v_n}_\alpha(x,i,c)+ (m_{ii}(x,v_n(x,i))-\alpha) \cJ^{v_n}_\alpha(x,i,c)\\& =  f(x,i)\,,\quad \text{a.e.}\,\, (x,i)\in\Rd\times \mathbb{S}\,,
\end{align*} where 
\begin{equation*}
f(x,i) = -[ c(x,i,v_n(x,i)) + \sum_{j \neq i} m_{ij}(x,v_n(x,i)) \cJ_{\alpha}^{v_n}(x,j,c)  ]\,.
\end{equation*}
Then using the standard elliptic PDE estimate as in \cite{GilTru}*{Theorem~9.11},  for any $p\geq d+1$ and $R >0$, we deduce that
\begin{equation}\label{ETC1.3A}
\norm{\cJ^{v_n}_\alpha(x,i,c)}_{\Sob^{2,p}(\sB_R)}
\,\le\, \kappa_1\bigl(\norm{\cJ^{v_n}_\alpha(x,i,c)}_{\Lp^p(\sB_{2R})} + \norm{f(x,i)}_{L^p(\sB_{2R})}\bigr)\,,
\end{equation}
where $\kappa_1$ is a positive constant which is independent of $n$\,.

Since 
\begin{align*}
&\norm{c(x,i,\zeta)}_{\infty} \,\df\, \sup_{(x,i, \zeta)\in\Rd\times \mathbb{S}\times\Act} c(x,i, \zeta) \leq M, \quad  \cJ_{\alpha}^{v_n}(x,i) \leq \frac{\norm{c(x,i, \zeta)}_{\infty}}{\alpha}\,,\\
&\qquad\qquad\text{and}\quad \norm{m_{ij}(x, \zeta)}_{\infty} \,\df\, \sup_{(x,\zeta)\in\Rd\times\Act} |m_{ij}(x,\zeta)| \leq M,
\end{align*}
we have,
\begin{align}\label{fb}
\norm{f(x,i)}_{L^p(\sB_{2R})}
&\leq \norm{c(x,i, \zeta)}_{L^p(\sB_{2R})} +\norm{\sum_{j\neq i}m_{ij}(x,v_n(x,i))\cJ_{\alpha}^{v_n}(x,j)}_{L^p(\sB_{2R})}\nonumber\\
&\leq \norm{c(x,i, \zeta)}_{\infty}|\sB_{2R}|^{\frac{1}{p}}\,+\,\norm{\sum_{j\neq i}m_{ij}(x,v_n(x,i))}_{\infty}\norm{\cJ_{\alpha}^{v_n}(x,j)}_{\infty}|\sB_{2R}|^{\frac{1}{p}}\nonumber\\
&\leq M|\sB_{2R}|^{\frac{1}{p}}\,+\,|\mathbb{S}|M.\frac{M}{\alpha}|\sB_{2R}|^{\frac{1}{p}}\nonumber\\
&\leq  M|\sB_{2R}|^{\frac{1}{p}}\bigl( 1+ \frac{|\mathbb{S}|M}{\alpha}\bigr)
\end{align}
from \cref{ETC1.3A} and \cref{fb}, we get
\begin{equation}\label{ETC1.3B}
\norm{\cJ_{\alpha}^{v_n}(x,i)}_{\Sob^{2,p}(\sB_R\times \mathbb{S})}=\sum_{i \in \mathbb{S}}\norm{\cJ_{\alpha}^{v_n}(x,i)}_{\Sob^{2,p}(\sB_R)}
\,\le\,\kappa_1 |\mathbb{S}| M |\sB_{2R}|^{\frac{1}{p}}\bigl(1+ \frac{M|\mathbb{S}|+1}{\alpha}\bigr)\,.
\end{equation}

We know that for $1< p < \infty$, the space \( \Sob^{2,p}(\cB_R \times \mathbb{S}) \) is reflexive and separable for \( 1 < p < \infty \); hence, as a corollary of the Banach Alaoglu theorem, we have that every bounded sequence in $\Sob^{2,p}(\sB_R\times \mathbb{S})$ has a weakly convergent subsequence (see, \cite{HB-book}*{Theorem~3.18}). Also, we know that for $p\geq d+1$ the space $\Sob^{2,p}(\sB_R)$ is compactly embedded in $\cC^{1, \beta}(\bar{\sB}_R)$\,, where $\beta < 1 - \frac{d}{p}$ (see \cite{ABG-book}*{Theorem~A.2.15 (2b)}). Since $\mathbb{S}$ is finite, the space $\Sob^{2,p}(\cB_R \times \mathbb{S})$ can be thought of as a finite product of such Sobolev spaces, one for each $i \in \mathbb{S}$. Therefore, the embedding $
\Sob^{2,p}(\cB_R \times \mathbb{S}) \hookrightarrow \cC^{1,\beta}(\cB_R \times \mathbb{S})$
is compact, which implies that every weakly convergent sequence in $\Sob^{2,p}(\sB_R\times \mathbb{S})$ will converge strongly in $\cC^{1, \beta}(\bar{\sB}_R\times \mathbb{S})$\,. Thus, in view of estimate \cref{ETC1.3B}, by standard diagonalization argument and Banach Alaoglu theorem, we can extract a subsequence $\{\cJ_{\alpha}^{v_{n_k}}\}$ such that for some $V_{\alpha}^*\in \Sobl^{2,p}(\Rd\times \mathbb{S})$
\begin{equation}\label{ETC1.3BC}
\begin{cases}
\cJ_{\alpha}^{v_{n_k}}\to & V_{\alpha}^*\quad \text{in}\quad \Sobl^{2,p}(\Rd\times \mathbb{S})\quad\text{(weakly)}\\
\cJ_{\alpha}^{v_{n_k}}\to & V_{\alpha}^*\quad \text{in}\quad \cC^{1, \beta}_{loc}(\Rd\times \mathbb{S}) \quad\text{(strongly)}\,.
\end{cases}       
\end{equation} 
Next, we will show that $V^*_{\alpha} = \cJ_{\alpha}^{v}$. Note that
\begin{align*}
&b(x,i,v_{n_k}(x,i))\cdot \grad \cJ_{\alpha}^{v_{n_k}}(x,i) - b(x,i,{v}(x,i))\cdot \grad {V}_{\alpha}^*(x,i) \\
&\qquad= b(x,i,v_{n_k}(x,i))\cdot \grad \left(\cJ_{\alpha}^{v_{n_k}} - {V}_{\alpha}^*\right)(x,i) + \left(b(x,i,v_{n_k}(x,i)) - b(x,i,{v}(x,i))\right)\cdot \grad {V}_{\alpha}^*(x,i)\,.
\end{align*}
Also,
\begin{align*}
    &\sum_{j \in \mathbb{S}} m_{ij}(x,v_{n_{k}}(x,i))\,\cJ_{\alpha}^{v_{n_k}}(x,j)\,-\, \sum_{j \in \mathbb{S}} m_{ij}(x,{v}(x,i))\,{V}^*_{\alpha}(x,j)\\&=\,\sum_{j \in \mathbb{S}} m_{ij}(x,v_{n_{k}}(x,i))\,(\cJ_{\alpha}^{v_{n_k}}(x,j)\,-\,{V}_{\alpha}^{*}(x,j))+ \sum_{j \in \mathbb{S}} (m_{ij}(x,v_{n_{k}}(x,i))-m_{ij}(x,{v}(x,i)))\,{V}_{\alpha}^{*}(x,j). 
\end{align*}
Since $\cJ_{\alpha}^{v_{n_k}}\to {V}_{\alpha}^*$ in $\cC^{1, \beta}_{loc}(\Rd\times \mathbb{S}),$ it follows that on every compact set $\left(b(x,i,v_{n_k}(x,i))\right)\cdot \grad \left(\cJ_{\alpha}^{v_{n_k}} - {V}_{\alpha}^*\right)(x,i)\to 0$ and $m_{ij}(x,v_{n_{k}}(x,i))\,(\cJ_{\alpha}^{v_{n_k}}(x,j)\,-\,{V}_{\alpha}^{*}(x,j))\to 0$\, strongly (since $m_{ij}$'s are bounded). Moreover, by the topology of $\Usm$, we have \begin{align*}
\left(b(x,i,v_{n_k}(x,i)) - b(x,i,{v}(x,i))\right)\cdot \grad {V}_{\alpha}^*(x,i)\to 0\,\quad\, \quad&\text{weakly}\\
\sum_{j \in \mathbb{S}} m_{ij}(x,v_{n_{k}}(x,i))\,\cJ_{\alpha}^{v_{n_k}}(x,j)\,-\, \sum_{j \in \mathbb{S}} m_{ij}(x,{v}(x,i))\,V_{\alpha}^*(x,j)\to 0\, \quad&\text{weakly}
\end{align*}
Thus, in view of the topology of $\Usm$, and the convergence $\cJ_{\alpha}^{v_{n_k}}\to {V}_{\alpha}^*$ in $\cC^{1, \beta}_{loc}(\Rd\times \mathbb{S})\,,$ as $k\to \infty$ we obtain 
\begin{align}\label{ETC1.4EA}
b(x,i,v_{n_k}(x,i))\cdot \grad \cJ_{\alpha}^{v_{n_k}}(x,i) + c(x,i,v_{n_k}(x,i))\,+\,\sum_{j \in \mathbb{S}} m_{ij}(x,v_{n_{k}}(x,i))\,\cJ_{\alpha}^{v_{n_k}}(x,j)\,\nonumber\\ \to b(x,i,{v}(x,i))\cdot \grad {V}_{\alpha}^*(x,i) + c(x,i,{v}(x,i))\,+\,\sum_{j \in \mathbb{S}} m_{ij}(x,{v}(x,i))\,{V}_{\alpha}^{*}(x,j)\quad\text{weakly}\,.
\end{align}

Now, multiplying \cref{ETC1.4B} by a test function $\phi\in \cC_{c}^{\infty}(\Rd\times \mathbb{S})$ and integrating over \(\Rd\), we obtain
\begin{align*}
&\int_{\Rd}\trace\bigl(a(x,i)\grad^2 \cJ_{\alpha}^{v_{n_k}}(x,i)\bigr)\phi(x,i)\D x + \int_{\Rd}\{b(x,i,v_{n_k}(x,i))\cdot \grad \cJ_{\alpha}^{v_{n_k}}(x,i) +  c(x,i,v_{n_k}(x,i))\\
&+\,\sum_{j \in \mathbb{S}} m_{ij}(x,v_{n_{k}}(x,i))\,\cJ_{\alpha}^{v_{n_k}}(x,j)\}\phi(x,i)\D x 
 = \alpha\int_{\Rd} \cJ_{\alpha}^{v_{n_k}}(x,i)\phi(x,i)\D x\,.
\end{align*}
Hence, by \cref{ETC1.3BC}, \cref{ETC1.4EA}, and letting $k\to\infty$, we obtain
\begin{align}\label{ETC1.4EB}
&\int_{\Rd}\trace\bigl(a(x,i)\grad^2 {V}_{\alpha}^*(x,i)\bigr)\phi(x,i)\D x + \int_{\Rd} \{b(x,i,{v}(x,i))\cdot \grad {V}_{\alpha}^*(x,i) + c(x,i,{v}(x,i))\nonumber\\
&+\,\sum_{j \in \mathbb{S}} m_{ij}(x,{v}(x,i))\,{V}_{\alpha}^{*}(x,j)\}\phi(x,i)\D x\,=\, \alpha\int_{\Rd} {V}_{\alpha}^*(x,i)\phi(x,i)\D x\,.
\end{align}
Since $\phi\in \cC_{c}^{\infty}(\Rd\times \mathbb{S})$ is arbitrary and ${V}_{\alpha}^*\in \Sobl^{2,p}(\Rd\times \mathbb{S})$, it follows from \cref{ETC1.4EB} that 
the function ${V}_{\alpha}^*\in \Sobl^{2,p}(\Rd\times \mathbb{S})\cap \cC_{b}(\Rd\times \mathbb{S})$ satisfies
\begin{align}\label{ETC1.4F}
&\trace\bigl(a(x,i)\grad^2 {V}_{\alpha}^{*}(x,i)\bigr) + b(x,i,{v}(x,i))\cdot \grad {V}_{\alpha}^{*}(x,i) + c(x,i,{v}(x,i))+\,\sum_{j \in \mathbb{S}} m_{ij}(x,{v}(x,i))\,{V}_{\alpha}^{*}(x,j)\nonumber\\ &= \alpha {V}_{\alpha}^{*}(x,i)\,
\end{align}

Let $(\tilde{X}, \tilde{S})$ be the solution of the SDE \cref{E1.1} corresponding to $v$. Then, by the It${\rm o}$–Krylov formula (\cite{ABG-book}*{Lemma~5.1.4}), we deduce the following.
\begin{align*}
&\Exp_{x,i}^{v}\left[ e^{-\alpha T} V_{\alpha}^{*}(\tilde{X}_{T},\tilde{S}_{T})\right] - V_{\alpha}^{*}(x,i)\\
&\,=\,\Exp_{x,i}^{v}\Bigg[\int_0^{T} e^{-\alpha s}\{\trace\bigl(a(\tilde{X}_s,\tilde{S}_s)\grad^2 V_{\alpha}^{*}(\tilde{X}_s,\tilde{S}_s)\bigr) + b(\tilde{X}_s,\tilde{S}_s, v(\tilde{X}_s,\tilde{S}_s))\cdot \grad V_{\alpha}^{*}(\tilde{X}_s,\tilde{S}_s)\\
& \quad+ \sum_{j \in \mathbb{S}} m_{\tilde{S}_sj}(\tilde{X}_s,v(\tilde{X}_s,\tilde{S}_s))\,V_{\alpha}^*(\tilde{X}_s,j) - \alpha V_{\alpha}^{*}(\tilde{X}_s,\tilde{S}_s)\} \D{s}\Bigg]
\end{align*}

Hence, using \cref{ETC1.4F} and rewriting the above equation, we obtain,
\begin{align}\label{ETC1.3FC}
e^{-\alpha T}\Exp_{x,i}^{v}\left[  V_{\alpha}^{*}(\tilde{X}_T,\tilde{S}_{T})\right] - V_{\alpha}^{*}(x,i) \,=\,- \Exp_{x,i}^{v}\left[\int_0^{T} e^{-\alpha s}c(\tilde{X}_s,\tilde{S}_s, v(\tilde{X}_s,\tilde{S}_s))\D{s}\right] \,.
\end{align}
Since $V_{\alpha}^{*}$ is bounded, 
it follows that 
\(
e^{-\alpha T}\Exp_{x,i}^{v}\left[  V_{\alpha}^{*}(\tilde{X}_T,\tilde{S}_{T})\right] \to 0
\)  as $T\to\infty$.
Now, by monotone convergence theorem and letting $T \to \infty$ in \cref{ETC1.3FC} we obtain,
\begin{align}\label{ETC1.3FD}
 V_{\alpha}^{*}(x,i) \,=\, \Exp_{x,i}^{v}\left[\int_0^{\infty} e^{-\alpha s}c(\tilde{X}_s,\tilde{S}_s, v(\tilde{X}_s,\tilde{S}_s)) \D{s}\right]=\cJ_{\alpha}^v(x,i,c) \,
\end{align}
Since every subsequence of $\{\cJ_{\alpha}^{v_n}\}$ admits a further subsequence converging to the unique solution $\cJ_{\alpha}^{v}$, every convergent subsequence has the same limit. Hence the entire sequence $\{\cJ_{\alpha}^{v_n}\}$ converges to $\cJ_{\alpha}^{v}$.
 This completes the proof.
\end{proof}

\subsection{Continuity of the ergodic cost function}\hspace{50em}
\vspace{0.5em}

We now consider the ergodic (long-run average) cost problem for the controlled 
RSDP model introduced in \cref{Cost}.

Throughout this subsection, the analysis is carried out under a Lyapunov stability condition.

\subsubsection{Under Lyapunov stability}\label{Lyapunov stability}

We impose the following Foster--Lyapunov condition on the dynamics.
\begin{itemize}
\item[\hypertarget{A6}{(A6)}]
There exists a positive constant $\widehat{C}_0$, and a pair of inf-compact  functions $(\Lyap, h)\in \cC^{2}(\Rd\times \mathbb{S})\times\cC(\Rd\times \mathbb{S} \times \Act)$ (i.e., the sub-level sets $\{\Lyap(\cdot,i)\leq k\} \,,\{h(\cdot,i,\cdot)\leq k\}$ are compact or empty sets in $\Rd$\,, $\Rd \times \Act$ respectively for each $k\in\RR,\,i\in \mathbb{S}$) such that
\begin{equation}\label{Lyap1}
\sL_{\zeta}\Lyap(x,i) \leq \widehat{C}_{0} - h(x,i, \zeta)\quad \forall\,\,\, (x,i,\zeta)\in \Rd\times \Act\,,
\end{equation} where $h$ is locally Lipschitz continuous in its first argument, uniformly with respect to the remaining variables.

\end{itemize} 

By \cite{rathia}*{Theorem 4.4}, we obtain existence and uniqueness of solutions to the Poisson equation associated with any fixed stationary Markov control.

\begin{theorem}\label{TErgoExisPoiss1}
Suppose that assumptions \hyperlink{A1}{(A1)}- \hyperlink{A6}{(A6)} hold. Then, for each $v\in \Usm$ there exists a unique  pair $(V^v, \rho^{v})\in \Sobl^{2,p}(\Rd\times \mathbb{S})\cap \sorder(\Lyap)\times \RR$ for any $p >1$ satisfying
\begin{equation}\label{TErgoExisPoiss1A}
\rho^{v} = \sL_{v}V^v(x,i) + c(x,i, v(x,i))\quad\text{with}\quad V^v(0,1) = 0\,.
\end{equation}
Furthermore, the following hold:
\begin{itemize}
\item[(i)]$\rho^{v} = \sE_{x,i}(c, v)$
\item[(ii)] for all $(x,i)\in\Rd \times \mathbb{S}$, we have
\begin{equation}\label{TErgoExisPoiss1B}
V^v(x,i) \,=\, \lim_{r\downarrow 0}\Exp_{x,i}^{v}\left[\int_{0}^{\uuptau_{r}} \left( c(X_t,S_t, v(X_t,S_t)) - \sE_{x,i}(c, v)\right)\D t\right]\,.
\end{equation}
\end{itemize} 
\end{theorem}

We now turn to the continuity result of the ergodic cost.
\begin{theorem}\label{ErgodLyapRobu1}
Suppose that Assumptions \hyperlink{A1}{(A1)}-\hyperlink{A6}{(A6)} hold. Then the map v $\mapsto \sE(v)$ from $\Usm$ to $\RR$ is continuous. i.e., for a sequence of policies \(\{v_n\}_n\) in \(\Usm\) satisfying \(v_n\to v\in\Usm\),
\begin{equation*}\label{ErgodLyapRobu1A}
\lim_{n\to\infty} \inf_{(x,i)\in\Rd\times \mathbb{S}}\sE_{x,i}(c, v_{n}) = \inf_{(x,i)\in\Rd\times \mathbb{S}}\sE_{x,i}(c, v)\,.
\end{equation*}
\end{theorem}
\begin{proof} Let $v_n$ be a sequence in $\Usm$ such that \(v_n\to v\) in $\Usm$. From \cref{TErgoExisPoiss1}, for each $n \in \NN$, there exists a unique pair $(V^{v_{n}}, \rho^{v_{n}})\in \Sobl^{2,p}(\Rd\times \mathbb{S})\cap\sorder{(\Lyap)}\times \RR$, $1< p < \infty$, with $V^{v_{n}}(0,1) = 0$, satisfying
\begin{equation}\label{ErgodLyapRobu1B}
\rho^{v_{n}} = \sL_{v_{n}}V^{v_{n}}(x,i) + c(x,i,{v_{n}}(x,i))
\end{equation}
In view of \cite{rathia}*{Theorem A.2 (1) and (2)}, from \cite{rathia}*{Theorem A.1}, there exists a constant $\hat{\kappa}_1 > 0$, independent of $n \in \NN$, such that \(\norm{V^{v_{n}}}_{\Sob^{2,p}(\sB_R\times \mathbb{S})}\leq \hat{\kappa}_1\). By the Banach–Alaoglu theorem and standard diagonalization argument (as in \cref{ETC1.3BC}), we deduce the existence of $\widehat{V}^*\in \Sobl^{2,p}(\Rd\times \mathbb{S})$ such that along a subsequence
\begin{equation*}\label{ErgodLyapRobu1C}
\begin{cases}
V^{v_{n_k}}\to & \widehat{V}^*\quad \text{in}\quad \Sobl^{2,p}(\Rd\times \mathbb{S})\quad\text{(weakly)}\\
V^{v_{n_k}}\to & \widehat{V}^*\quad \text{in}\quad \cC^{1, \beta}_{loc}(\Rd\times \mathbb{S})\quad\text{(strongly)}\,.
\end{cases}       
\end{equation*}
for some $0 < \beta < 1 - \tfrac{d}{p}$.
Since $\rho^{v_{n}} \leq M$, there exists a further subsequence (denoted by the same index) such that
$\rho^{v_{n_k}}\to \hat{\rho}$ as $k\to \infty$\,.
Multiplying \cref{ErgodLyapRobu1B} by a test function, integrating over \(\Rd\) and passing to the limit $k \to \infty$, it follows that $(\widehat{V}^*, \hat{\rho})\in \Sobl^{2,p}(\Rd\times \mathbb{S})\times \RR$, \, $1< p < \infty$ satisfies 
\begin{equation}\label{ErgodLyapRobu1D}
\hat{\rho} = \sL_{v}\widehat{V}^*(x,i) + c(x,i,{v}(x,i))
\end{equation}
Since $V^{v_{n_k}}(0,1) = 0$ for all $k \in \NN$, we have $\widehat{V}^*(0,1) = 0$.

Next, arguing similar to the proof of \cite{rathia}*{Theorem 4.3} one can show that $\widehat{V}^*\in \sorder{(\Lyap)}$.
Since $(\widehat{V}^*, \hat{\rho})\in \Sobl^{2,p}(\Rd\times \mathbb{S})\cap \sorder(\Lyap)\times \RR$ satisfies $\widehat V^*(0,1) = 0$ and the ergodic HJB equation \cref{ErgodLyapRobu1D}, the uniqueness result of \cite{rathia}*{Theorem 3.4} implies that
 $(\widehat{V}^*, \hat{\rho}) \equiv (V^v, \rho^v)$.
 Therefore every convergent subsequence of
$\{(V^{v_n},\rho^{v_n})\}$ has the same limit
$(V^{v},\rho^{v})$.
Consequently,
$\rho^{v_n}\to\rho^{v}$, and hence
$\sE(v_n)\to\sE(v)$.
 This completes the proof of the theorem.

\end{proof}

\subsection{Continuity of the finite horizon cost function}\hspace{50em}
\vspace{0.5em}

In this subsection, we study the finite-horizon cost and establish its continuity. 
Unlike the discounted and ergodic cases, the corresponding value function depends explicitly on time and satisfies a parabolic system of coupled Hamilton--Jacobi--Bellman equations with a prescribed terminal condition.

\begin{theorem}\label{thm:finite-horizon-robust}
Suppose Assumptions \hyperlink{A1}{(A1)}-\hyperlink{A3}{(A3)} and \hyperlink{A5}{(A5)}  hold.  Then the map  $v\mapsto\cJ_T(v)$ from $\Um$ to $\RR$ is continuous. i.e., for running cost $c$, 
\[
\lim_{n\to\infty} \cJ_T^{v_n}(x,i,c) = \cJ_T^v(x,i,c) \qquad \forall\,\, (x,i)\in \Rd\times\mathbb{S}.
\]
\end{theorem}

\begin{proof}
Let $v_n$ be a sequence in \(\Um\) such that
 $v_n \to v$ in $\Usm$. From \cite{DTSB}*{Theorem 1} for each $n\in \NN$ there exists a unique solution \({\psi}_n\in \Sobl^{1,2,p}((0,T)\times\Rd\times\mathbb{S})\) to the following Parabolic PDE equation with terminal data:
\begin{equation}\label{FH5.13}
\begin{aligned}
\partial_t {\psi}_n(t,x,i) 
+ \sL_{v_n} {\psi}_n(t,x,i) + c(x,i,v_n(t,x,i)) &= 0, \quad\forall\, (t,x,i)\in (0,T)\times \Rd\times \mathbb{S}\\
{\psi}_n(T,x,i) &= c_{_{T}}(x,i),\quad\forall\, (x,i)\in \Rd\times \mathbb{S}
\end{aligned}
\end{equation}
Thus, by the parabolic PDE estimates \cite{DTSB}*{Theorem 1}, for any $p>d+2$ and $R>0$, the solution of \cref{FH5.13} satisfies
\begin{equation}\label{FHestm}
\|\psi_n\|_{\Sob^{1,2,p}((0,T)\times\sB_R\times\mathbb{S})} \le \tilde{\kappa}_1 (1+\tilde{\kappa}+\|c\|_{\infty}|\sB_{2R}|^{\frac{1}{p}}+\|c_{_{T}}\|_\infty). 
\end{equation}
for some $\tilde{\kappa}_1,\, \tilde{\kappa}>0$.
Thus, from \eqref{FHestm}, we obtain
\begin{equation}\label{FHapestm}
\|\psi_n\|_{\Sob^{1,2,p}((0,T)\times\sB_R\times\mathbb{S})} \le \tilde{\kappa}_2
\end{equation}
for some positive constant $\tilde{\kappa}_2$ independent of $n$. Since $\Sob^{1,2,p}((0,T)\times\sB_R\times\mathbb{S})$ is a reflexive Banach space, in view of \cref{FHapestm}, compact embedding $\Sob^{1,2,p}((0,T)\times\sB_R\times\mathbb{S}) \hookrightarrow \Sob^{0,1,p}((0,T)\times\sB_R\times\mathbb{S})$ and by the arguments as in \cref{{ETC1.3B}}-\cref{ETC1.3BC} there exists $\hat{\psi} \in \Sobl^{1,2,p}((0,T)\times\Rd\times \mathbb{S})$ such that (along a subsequence, denoted by the same sequence)
\begin{equation}\label{FHcgs}
\begin{aligned}
\psi_n &\to \hat{\psi} \quad \text{in } \Sobl^{1,2,p}((0,T)\times\Rd\times \mathbb{S})\quad\text{(weakly)} \\
\psi_n &\to \hat{\psi} \quad \text{in } \Sobl^{0,1,p}((0,T)\times\Rd\times \mathbb{S})\quad\text{(strongly)}.
\end{aligned}
\end{equation}
Multiplying both sides of \eqref{FH5.13} by a test function $\varphi \in \cC_c^\infty((0,T)\times \Rd\times \mathbb{S})$ and integrating, we obtain
\begin{equation}\label{FH3.5}
\int_0^T \int_{\Rd} \partial_t\psi_n(t,x,i)\,\varphi(t,x,i)\,dxdt + 
\int_0^T \int_{\Rd}  \big[\sL_{v_n} \psi_n(t,x,i) + c(x,i,v_{n}(t,x,i))\big]\varphi(t,x,i)\,dxdt = 0.
\end{equation}
Thus, in view of  \cref{FHcgs} and by letting $n\to\infty$,
from \cref{FH3.5}, it follows 
(by arguments similar to those in \cref{ETC1.3BC}--\cref{ETC1.4F}) that 
$\hat{\psi}$ satisfies the limiting HJB equation
\begin{equation}\label{abcd}
\begin{aligned}
\partial_t \hat{\psi}(t,x,i) +\sL_{v} \hat{\psi}(t,x,i) + c(x,i,v(t,x,i)) &= 0,\quad\forall\, (t,x,i)\in (0,T)\times \Rd\times \mathbb{S} \\
\hat{\psi}(T,x,i) &= c_{_{T}}(x,i),\quad\forall\, (x,i)\in \Rd\times \mathbb{S} 
\end{aligned}
\end{equation}

Now applying the It$\hat{\rm o}$–Krylov formula (\cite{ABG-book}*{Lemma~5.1.4}) in \cref{abcd}, we deduce that
\begin{align}\label{FH5.17}
\hat{\psi}(t,x,i) 
= \Exp_{x,i}^{v}\!\left[ \int_t^T c(X_s,S_s, v(s,X_s,S_s))\,ds + c_{_{T}}(X_T,S_T) \right]. 
\end{align}

Hence, from \cref{FH5.17}, we conclude that 
\[
\hat{\psi}(0,x,i)=\cJ^v_T(x,i,c).
\]
Since every convergent subsequence of $\{\psi_n\}$ converges to the unique solution $\hat \psi$ of \eqref{abcd}, the whole sequence converges to $\hat\psi$.
Hence,
\[
\psi_n(0,x,i)\to\hat\psi(0,x,i)=\cJ_T^{v}(x,i,c).
\]
This completes the proof.
\end{proof}

\subsection{Continuity of the exit time cost function}\hspace{50em}
\vspace{0.5em}

Finally, we consider the exit-time cost criterion. 
We assume that $\beta \in \cC(\bar{\mathcal{O}}\times\mathbb{S}\times\Act)$ and $h \in \Sob^{2,p}(\mathcal{O}\times\mathbb{S})$. Following the derivation in \cite{VSB2005}*{p. 228-229}, the associated HJB equation is given by
\[
\min_{\zeta \in \Act} \left[\sL_\zeta \varphi(x,i) - \beta(x,i,\zeta)\,\varphi(x,i) + c(x,i,\zeta)\right] = 0, 
\quad (x,i) \in \mathcal{O}\times\mathbb{S}, 
\qquad \varphi = h \text{ on } \partial \mathcal{O}\times\mathbb{S}.
\]
The following theorem establishes continuity of the exit-time cost.

\begin{theorem}\label{thm:exit-time-continuity}
Suppose Assumptions \hyperlink{A1}{(A1)}-\hyperlink{A3}{(A3)} and \hyperlink{A5}{(A5)} hold.  Then the map  $v\mapsto\hat\cJ_e$ from $\Usm$ to $\RR$ is continuous, i.e., for running cost $c$,
\[
\lim_{n\to\infty} \hat\cJ_e^{v_n}(x,i,c) = \hat\cJ_e^v(x,i,c) \qquad \forall\,\, (x,i)\in \Rd\times\mathbb{S}.
\]
\end{theorem}

\begin{proof}
Let $v_n$ be a sequence in \(\Um\) such that
 $v_n \to v$. From \cite{AS}*{Theorem 2.1} for each $n\in \NN$ there exists a unique solution \(\hat{\cJ}^{v_n}_e\in \Sob^{2,p}(\mathcal{O}\times\mathbb{S})\) to the following equation:
    \begin{equation}\label{exit-time-poisson}
     \begin{aligned}
    \sL_{v_n} \hat{\cJ}^{v_n}_e(x,i) - \beta(x,i,{v_n}(x,i))\,\hat{\cJ}^{v_n}_e(x,i) + c(x,i,v_n(x,i)) &= 0, 
\quad (x,i) \in \mathcal{O}\times\mathbb{S}, \\
 \hat{\cJ}^{v_n}_e &= h \quad\text{ on } \partial \mathcal{O}\times\mathbb{S}.
      \end{aligned}
 \end{equation}
 Standard elliptic estimates imply uniform boundedness of \(\|\hat{\cJ}^{v_n}_e\|_{\Sob^{2,p}(\mathcal{O}\times\mathbb{S})}\), independent of \(n\).  
By the Banach–Alaoglu theorem and standard diagonalization argument, there exists a subsequence $\hat{\cJ}^{v_{n_{k}}}_e$ and a limit function \(\bar{\cJ}_e\) such that
\[
\hat{\cJ}^{v_{n_{k}}}_e \to \bar{\cJ}_e \text{ in } \Sob^{2,p} (\mathcal{O}\times\mathbb{S})\quad\text{weakly}, \qquad
\hat{\cJ}^{v_{n_{k}}}_e \to \bar{\cJ}_e \text{ in } \cC^{1,\beta}(\mathcal{O}\times\mathbb{S})\quad \text{strongly.}
\]
Multiplying \cref{exit-time-poisson} by a test function 
\(\phi\in \cC_c^\infty(\mathcal{O}\times\mathbb{S})\) and integrating over \(\mathcal{O}\), then  
passing to the limit as \(n\to\infty\), using the 
the strong convergence  of \(\hat{\cJ}^{v_{n_k}}_e\) in \(\cC^{1,\beta}(\mathcal{O}\times\mathbb{S})\), and the convergence \(v_n\to v\) in \(\Usm\), we obtain
 \(\bar{\cJ}_e\) satisfies
\[\sL_{v} \bar{\cJ}_e(x,i)
   - \beta(x,i,v(x,i))\bar{\cJ}_e(x,i)
   + c(x,i,v(x,i)) = 0\quad in\;\; \mathcal{O}\times\mathbb{S},
   \qquad \bar{\cJ}_e = h \text{ on } \partial\mathcal{O}\times\mathbb{S}.
\]
By applying It\^o-Krylov formula (\cite{ABG-book}*{Lemma 5.1.4}) to the above Dirichlet equation we obtain, \(\bar{\cJ}_e(x,i)=\hat \cJ_e^v(x,i)\) for all \((x,i)\in\mathcal{O}\times\mathbb{S}.\) Since every convergent subsequence of
$\{\hat{\cJ}^{v_n}_e\}$
has the same limit
$\hat{\cJ}^{v}_e$,
the entire sequence converges to
$\hat{\cJ}^{v}_e$.
Therefore,
\[
\hat{\cJ}^{v_n}_e(x,i)\to
\hat{\cJ}^{v}_e(x,i).
\]
\end{proof}

\section{ Denseness of finite action/piecewise constant/Lipschitz continuous Markov/stationary Markov policies}\label{Denseness}
In this section, we show that the finite action/piecewise constant/Lipschitz continuous Markov/stationary Markov policies are dense in the space of Markov/stationary Markov policies.
\subsection{Denseness of finite action stationary policies}\label{finite-action-construction}\hspace{50em}
\vspace{0.5em}

Let $d_{\Act}$ be the metric on the compact action space $\Act$. Since $\Act$ is compact, it is totally bounded. Hence, for each \(n\in\NN\) there exists a finite grid \(\{\zeta_{n,k}\}_{k=1}^{m_n}\) such that 
\[\min_{k=1\dots m_n}d_{\Act}(\zeta,\zeta_{n,k})<\frac{1}{n},\qquad\forall\;\zeta\in\Act\]
Set \(\Lambda_n\df\{\zeta_{n,1}\dots\zeta_{n,m_n}\}\) and the nearest–neighbour quantizer (see, \cite{MR3722422})  \(\mathcal{Q}_n:\Act\to\Lambda_n\) by \[\mathcal{Q}_n(\zeta)=\argmin_{\zeta_{n,k}\in\Lambda_n}d_{\Act}(\zeta,\zeta_{n,k}),\]
where ties are broken by choosing the smallest index, so that $\mathcal{Q}_n$ is measurable.
The map 
\(\mathcal{Q}_n\) induces a measurable partition  $\{\Act_{n,k}\}_{k=1}^{m_n}$ on $\Act$, where \[\Act_{n,k}=\{\zeta\in\Act \;|\; \mathcal{Q}_n(\zeta)=\zeta_{n,k}\}.\]
and by the triangle inequality $\diam(\Act_{n,k})<\frac{2}{n}$. Now, for each \(v\in \Usm\) define the associated finite–action policy \(v_n\) by
\begin{equation}\label{df-finiteaction}v_n(\zeta_{n,k}\;|\;(x,i))\df v(\mathcal{Q}_n^{-1}(\zeta_{n,k})\;|\;(x,i))=v(\Act_{n,k}\;|\;(x,i)).\end{equation}
Thus $ v_n$ takes values in the finite set 
\(\Lambda_n\).

From \cite{SPSY}*{Lemma 4.1} we have that the space of stationary markov policies with finite actions are dense in \(\Usm\) with respect to the Borkar topology.
\begin{lemma}
    For each \(v\in\Usm\) there exists a sequence $v_n$ (defined as in \cref{df-finiteaction}) of policies with finite actions, satisfying
    \begin{equation}
        \lim_{n\to\infty}\int_{\Rd}f(x,i)\int_{\Act}g(x,i,\cdot)v_{n}(x,i)(\D \zeta)\D x = \int_{\Rd}f(x,i)\int_{\Act}g(x,i,\cdot)v(x,i)(\D \zeta)\D x
\end{equation}
for all $f\in \Lp^1(\Rd\times \mathbb{S})\cap \Lp^2(\Rd\times \mathbb{S})$, $g\in \cC_b(\Rd\times \mathbb{S} \times  \Act)$ and $i\in \mathbb{S}$.
\end{lemma}

\subsection{Denseness of piecewise constant Markov/stationary Markov policies}\label{p.w_denseness}\hspace{50em}
\vspace{0.5em}

Let \(d_{\sP}\) denote the Prokhorov metric on the space $V$ of probability measures on $\Act$. Since \((\Act,d_{\Act})\) is compact, $V$ is separable and the convergence in \((V,d_{\sP})\) is equivalent to the weak convergence.
From \cite{SPSY}*{Theorem 4.2} we have that the space of piecewise constant policies are dense in \(\Usm\).
\begin{theorem}
    For each \(v\in\Usm\) there exists a sequence $v_n$ of piecewise constant stationary policies, satisfying
    \begin{equation}
        \lim_{n\to\infty}\int_{\Rd}f(x,i)\int_{\Act}g(x,i,\cdot)v_{n}(x,i)(\D \zeta)\D x = \int_{\Rd}f(x,i)\int_{\Act}g(x,i,\cdot)v(x,i)(\D \zeta)\D x
\end{equation}
for all $f\in \Lp^1(\Rd\times \mathbb{S})\cap \Lp^2(\Rd\times \mathbb{S})$, $g\in \cC_b(\Rd\times \mathbb{S} \times  \Act)$ and $i\in \mathbb{S}$.
\end{theorem}
Similarly, from \cite{SPSY}*{Theorem 6.2} the space of piecewise constant policies are dense in \(\Um\).
\begin{theorem}
    Let $v\in\Um$ then there exists a sequence $v_n$ of piecewise constant markov policies, satisfying
   \begin{align*}
\lim_{n \to \infty} 
&\int_0^\infty \int_{\Rd} f(t,x,i) 
\Bigg( \int_{\Act} g(t,x,i,\zeta)\, v_n(t,x,i)(d\zeta) \Bigg) dx\,dt\nonumber\\
&=
\int_0^\infty \int_{\Rd} f(t,x,i) 
\Bigg( \int_{\Act} g(t,x,i,\zeta)\, v(t,x,i)(d\zeta) \Bigg) dx\,dt,
\end{align*}
for all 
$i\in\mathbb{S,}\,\, f \in L^1([0,\infty)\times\Rd \times\mathbb{S}) \cap L^2([0,\infty)\times\Rd \times\mathbb{S} ),\,\, g \in \cC_b([0,\infty)\times\Rd\times\mathbb{S} \times \Act).$

\end{theorem}
\subsection{Denseness of Lipschitz continuous Markov/stationary Markov policies}\hspace{50em}
\vspace{0.5em}

Finally, from \cite{SPSY2}*{Theorem 3.1} and \cite{SPSY2}*{Theorem 3.2}, we have that the space of Lipschitz stationary policies are dense in $\Usm/\Um$.
\begin{theorem}
    For each \(v\in\Usm\) there exists a sequence $v_n$ of Lipschitz policies in $\Usm$, satisfying
    \begin{equation}
        \lim_{n\to\infty}\int_{\Rd}f(x,i)\int_{\Act}g(x,i,\cdot)v_{n}(x,i)(\D \zeta)\D x = \int_{\Rd}f(x,i)\int_{\Act}g(x,i,\cdot)v(x,i)(\D \zeta)\D x
\end{equation}
for all $f\in \Lp^1(\Rd\times \mathbb{S})\cap \Lp^2(\Rd\times \mathbb{S})$, $g\in \cC_b(\Rd\times \mathbb{S} \times  \Act)$ and $i\in \mathbb{S}$.
\end{theorem}
\begin{theorem}
    For each \(v\in\Um\) there exists a sequence $v_n$ of Lipschitz policies in $\Um$, satisfying
    \begin{align}
\lim_{n\to\infty}\int_0^\infty\int_{\Rd}f(t,x,i)\int_{\Act}g(t,x,i,\cdot)v_{n}(t,x,i)(\D \zeta)\D x\D t \nonumber \\ = \int_0^\infty\int_{\Rd}f(t,x,i)\int_{\Act}g(t,x,i,\cdot)v(t,x,i)(\D \zeta)\D x\D t
\end{align}
for all $f\in \Lp^1([0,\infty)\times\Rd\times \mathbb{S})\cap \Lp^2([0,\infty)\times\Rd\times \mathbb{S})$, $g\in \cC_b([0,\infty)\times\Rd\times \mathbb{S} \times  \Act)$ and $i\in \mathbb{S}$.
\end{theorem}
\section{ Near optimality of quantized/piece-wise constant/smooth policies for controlled RSDPs}\label{Section-near-op}

We show that the classes of finite-action, piecewise-constant, and Lipschitz policies provide $\epsilon$-optimal approximations for the optimal control problems under various cost criteria.

\subsection{Discounted Cost}
\begin{theorem}
    Suppose assumptions \hyperlink{A1}{(A1)}-\hyperlink{A3}{(A3)} and \hyperlink{A5}{(A5)} hold. Then for every $\epsilon>0$ there exists a finite action policy \(v_{\epsilon}^*\), a piecewise constant policy \(\bar v_{\epsilon}^{*}\), and a Lipschitz policy \(\hat v_{\epsilon}^*\)  in \(\Usm\) such that
    \begin{equation}\label{nearopt-discounted}
    \cJ_{\alpha}^{v_{\epsilon}^*}(x,i)\le V_\alpha(x,i)+\epsilon,\quad\cJ_{\alpha}^{\bar v_{\epsilon}^*}(x,i)\le V_\alpha(x,i)+\epsilon\quad \text{and}\quad \cJ_{\alpha}^{\hat v_{\epsilon}^*}(x,i)\le V_\alpha(x,i)+\epsilon
    \end{equation}
    for all $(x,i)\in\Rd\times\mathbb{S}$.    
\end{theorem}
\begin{proof}
By \cite{AGM93}*{Theorems~6.1-6.2 and Corollary~6.1}, there exists an optimal control  \(v^*\in \Usm\) satisfying \(\cJ_{\alpha}^{v^*}(x,i, c) = \inf_{U\in \Uadm}\cJ_{\alpha}^{U}(x,i, c)\;\forall\;(x,i)\in\Rd\times\mathbb{S}\). Since finite-action, piecewise-constant, and Lipschitz policies are dense in $\Usm$, and the map \(v\mapsto\cJ_\alpha(v)\) is continuous, each class contains an $\epsilon$-optimal policy, yielding \eqref{nearopt-discounted}. 
\end{proof}
\subsection{Exit-Time Cost}

\begin{theorem}
    Suppose assumptions \hyperlink{A1}{(A1)}-\hyperlink{A3}{(A3)} and \hyperlink{A5}{(A5)} hold. Then for every $\epsilon>0$ there exists a finite action \(v_{\epsilon}^*\), a piecewise constant  \(\bar v_{\epsilon}^{*}\), and a Lipschitz policy \(\hat v_{\epsilon}^*\)  in \(\Usm\) which is $\epsilon$-optimal for the exit-time cost. That is,
    \begin{equation}\label{nearopt-exit}
    \hat \cJ_{e}^{v_{\epsilon}^*}(x,i)\le \hat \cJ_{e}^{*}(x,i)+\epsilon,\quad\hat \cJ_{e}^{\bar v_{\epsilon}^*}\le \hat \cJ_{e}^{*}(x,i)+\epsilon\quad \text{and}\quad \hat \cJ_{e}^{\hat v_{\epsilon}^*}(x,i)\le \hat \cJ_{e}^{*}(x,i)+\epsilon\end{equation}
   for all \( (x,i)\in\Rd\times\mathbb{S}\).
\end{theorem}
\begin{proof}
By \cite{rathia}*{Theorem~6.1}, an optimal stationary Markov control $v^*\in\Usm$ exists. 
The result follows from the density of the structured policy classes in $\Usm$ together with continuity of the mapping $v\mapsto \hat{\cJ}_e^{v}$.
\end{proof}
Under the Lyapunov stability assumption, near optimality also holds for the ergodic criterion.

\subsection{Ergodic Cost}
\begin{theorem}
    Suppose the assumptions \hyperlink{A1}{(A1)}- \hyperlink{A6}{(A6)} hold. Then for every $\epsilon>0$ there exist finite-action \(v_{\epsilon}^*\), piecewise-constant \(\bar v_{\epsilon}^{*}\), and a Lipschitz policy \(\hat v_{\epsilon}^*\) in $\Usm$ which is $\epsilon$-optimal for the ergodic cost. That is,
    \begin{equation}\label{nearopt-ergodic}
         \sE_{x,i}(c,v_{\epsilon}^*)\le \sE^{*}(c)+\epsilon,\quad\sE_{x,i}(c,\bar v_{\epsilon}^*)\le \sE^{*}(c)+\epsilon\quad \text{and}\quad \sE_{x,i}(c,\hat v_{\epsilon}^*)\le \sE^{*}(c)+\epsilon
         \end{equation}
         for all \((x,i)\in\Rd\times\mathbb{S}\)
\end{theorem}
\begin{proof}
By the \cite[Theorems~A.2--A.3]{rathia}, an optimal stationary Markov control exists. 
The conclusion again follows from the density of the structured policies and continuity of $v\mapsto \sE^{v}(c)$.
\end{proof}

\subsection{Finite-Horizon Cost}\hspace{50em}
\vspace{0.5em}

For the finite-horizon case, we have near optimality of piecewise constant and Lipschitz Markov policies.
\begin{theorem}\label{nearop-finite horizon}
    Suppose the assumption \hyperlink{A1}{(A1)}-\hyperlink{A3}{(A3)} and \hyperlink{A5}{(A5)} holds. Then for every $\epsilon>0$ there exists a piecewise constant policy \(\bar v_{\epsilon}^{*}\), and a Lipschitz policy \(\hat v_{\epsilon}^*\) in \(\Um\) satisfying
    \begin{equation}\label{nearopt-finite} \cJ_{T}^{\bar v_{\epsilon}^*}(x,i)\le \cJ_{T}^{*}(x,i)+\epsilon\quad \text{and}\quad \cJ_{T}^{\hat v_{\epsilon}^*}(x,i)\le \cJ_{T}^{*}(x,i)+\epsilon \end{equation}
   for all \( (x,i)\in\Rd\times\mathbb{S}\).  
\end{theorem}
\begin{proof}
By \cite[Theorem~5.1]{rathia}, an optimal Markov control exists. 
The result follows from the density and the continuity of $v\mapsto \cJ_T^{v}$.

\begin{remark}
Finite-action policies provide a practical bridge between continuous control problems and implementable algorithms. By reducing the action space to a finite set, they enable tractable numerical optimization, dynamic programming approximations, and reinforcement learning implementations while retaining near-optimal performance (e.g., see \cite{MR3722422}).
\end{remark}

\begin{remark}
The preceding results provide a concrete application of the density of piecewise-constant policies. 
In particular, discretization of controlled regime-switching diffusions under piecewise-constant controls yields a discrete-time model whose state process and value functions converge to their continuous-time counterparts. 
This demonstrates how structural density results translate into rigorous numerical approximations. 
Furthermore, when additional regularity such as Lipschitz continuity of policies is imposed, one may expect analogous approximation results for more challenging performance criteria, such as ergodic costs 
(see \cite{pradhan2025nearoptimalitydiscretetimeapproximations}*{Theorem 5.2}).
\end{remark}
\end{proof}
\section{Application: Markov chain approximations}\label{Application}

\subsection{The Markov chain approximation for the finite horizon cost}\hspace{30em}
\vspace{0.5em}

In this section, we exploit the density of piecewise-constant policies to construct a discrete-time approximation of the controlled RSDP \cref{E1.1}. 
We introduce an Euler--Maruyama discretization under piecewise-constant controls and prove that the corresponding finite-horizon cost converges to its continuous-time counterpart. 
Moreover, we show that optimal policies obtained from the discrete-time model are near-optimal for the continuous-time system.

To facilitate the numerical approximation, we impose the following stronger regularity assumptions.
Throughout this section, the coefficients are assumed to be globally Lipschitz and bounded. Also, we assume that the switching rates $m_{ij}$'s are independent of the actions.

\begin{itemize}

\item[\hypertarget{B1}{(B1)}]
\emph{Global Lipschitz continuity and boundedness.}
There exists $C_L,\, M_2>0 $ such that
\begin{align*}
&|b(x,i,\zeta)-b(y,i,\zeta)|^2
+\|\upsigma(x,i)-\upsigma(y,i)\|^2
+|m_{ij}(x)-m_{ij}(y)|^2
\le C_L |x-y|^2,\\
&|b(x,i,\zeta)|^2
+\|\upsigma(x,i)\|^2
+|m_{ij}(x)|^2
\le M_2
\end{align*}
for all $x,y\in\Rd$, $i,j\in\mathbb{S}$, and $\zeta\in\Act$.

\item[\hypertarget{B2}{(B2)}]
\emph{Cost regularity.}
The running cost $c$ and terminal cost $c_T$ are globally Lipschitz and uniformly bounded, i.e., there exist $C_{L_c},M_3>0$ such that for all \(x,y\in\Rd,\;i\in\mathbb{S}\;\;\text{and}\;\zeta,\zeta'\in\Act\) :
\begin{align*}
|c(x,i,\zeta)-c(y,i,\zeta')| 
&\le C_{L_c}(|x-y|+|\zeta-\zeta'|),
\qquad |c_T(x,i)-c_T(y,i)|\le C_{L_c}(|x-y|)
\\
|c(x,i,\zeta)| + |c_T(x,i)| &\le M_3.
\end{align*}

\end{itemize}
\begin{remark}
The assumption that the switching rates are independent of the control action $\zeta$ is essential for proving the convergence of the mismatch probability $\mathbb{P}(\theta^h \le T)$; see \cref{mismatch_bound}. Without this assumption, it is difficult to establish the required convergence estimate.
\end{remark}
\begin{remark}
We assume in \hyperlink{B2}{(B2)} that the running cost is Lipschitz continuous in the control action $\zeta$ in order to derive convergence rates for the finite-horizon cost. However, convergence of the finite-horizon cost still holds without this assumption, although no explicit rate can be obtained in that case.
\end{remark}
 We will show that the finite horizon cost can be approximated by the discrete time finite horizon cost by using Euler-Maruyama approximation.
\subsubsection{\bf Euler-Maruyama approximation
}
Let $h>0$ be the time step and define \(t_h\df\lfloor\frac{t}{h}\rfloor h\) where \(\lfloor\; \rfloor\) denotes the floor function. For a piecewise–constant control \(U^h\in\Uadm\), we approximate the state-dependent regime-switching SDE \cref{E1.1} using the Euler-Maruyama scheme.

The continuous component is given by
\begin{equation}\label{s-3.1}
d X^h_t= b(X_{t_h}^h,S_{t_h}^h,U_{t_h}^h)\D t+\upsigma(X_{t_h}^h,S_{t_h}^h)\D W_t,
\end{equation}
and the discrete component is defined through the Poisson random measure \(\cP(\D t,\D z)\) by
\begin{equation}\label{s-3.2}
S^h_t=i+\int_0^t\int_{\RR_+}h(X^h_{s_h},S^h_{s-},z)\cP(\D s,\D z),
\end{equation}
where $\cP(dt,dz)$ is the same Poisson random measure introduced in \cref{E1.1} to determine the process $S_t$ with $S_0=i$. We denote by $(X_t^h,S^h_t)$ the EM approximation of $(X_t,S_t)$ for some given $h$, with initial condition $(X^h_0,S^h_0)=(X_0,S_0)=(x,i)$.  Then, by the Skorokhod's representation \eqref{s-3.2}, for $\delta\downarrow 0$
\begin{equation}\label{s-3.2b}
  \mathbb{P}(S^h_{t_h+\delta}=k|S^h_{t_h}=j, \,X^h_{s_h},\, s\le t)=\begin{cases}
  m_{jk}(X^h_{t_h})\delta+o(\delta),& k\neq j,\\
  1+m_{jj}(X^h_{t_h})\delta+o(\delta),& k=j,
  \end{cases}
\end{equation}
Thus $S^h_{t}$ is a continuous-time pure jump process whose transition rate depends on the frozen state  $X_{t_h}^h$. 
Moreover, over each interval \([kh,(k+1)h)\), the dynamics of 
$X_t^h$ depend on the embedded chain $(S^h_{kh})_{k\geq 1}$ of the process $S^h_{t}$ and the control $U_{t_h}^h$.

Under assumption  \hyperlink{B1}{(B1)} existence and uniqueness of the solution to \eqref{s-3.1}-\eqref{s-3.2} follow by a standard stepwise construction on each interval $ [kh, (k+1)h)$, $k\geq 0$.

\begin{remark}
    In view of \cref{s-3.2b} one can approximate the transition probability matrix of \(S^h_{(k+1)h}\)   by \(I+hQ(x)\) when \(X^h_{kh}=x\) i.e., if
\(
\mathcal G_n:=\sigma\{(X_{kh}^h,S_{kh}^h):0\le k\le n\}
\)
be the filtration generated by the approximating process \cref{s-3.1}-\cref{s-3.2}. Then the regime component
$\{S_{kh}^h\}$ is a controlled Markov chain on $\mathbb{S}$ satisfying
\begin{equation}
\mathbb P\!\left(
S_{(k+1)h}^h=j
\mid
X_{kh}^h=x,\,
S_{kh}^h=i,\,
\mathcal G_n
\right)
=
p_{ij}(x),
\qquad i,j\in\mathbb{S},
\end{equation}
where the transition matrix is chosen by the local consistency condition
\begin{equation}
P(x)=\big(p_{ij}(x)\big)_{i,j\in\mathbb{S}}=I+hQ(x),
\end{equation}
that is, for $j\neq i$,
\[
p_{ij}(x)=h\,m_{ij}(x),
\qquad
p_{ii}(x)
=
1-h\sum_{j\neq i} m_{ij}(x),
\]
\end{remark}
The main difficulty in the analysis of the Euler--Maruyama approximation for state-dependent RSDPs, compared with the state-independent case, lies in estimating the mismatch between the true and numerical switching processes. In particular, a key step is to control
\begin{equation}\label{o-1}
\int_0^t \mathbb P(S_s\neq S_s^h)\,ds .
\end{equation}
This problem was treated in Shao \cite{ShaoEM}*{Lemma 3.2}; However, within the framework of Shao's approach, explicit convergence rates appear to be obtainable only in the special case of additive noise. To overcome this difficulty, we employ the technique developed in \cite{nguyen2025hybrid}*{Chapter 5}.

\begin{theorem}[Strong convergence of the Euler-Maruyama scheme]\label{t-3.3}
Assume \hyperlink{B1}{(B1)} holds. Let $(X_t,S_t)$  solve \eqref{E1.1} under \(U\in\Uadm\) and let $(X^h_t,S^h_t)$ be the solution of \eqref{s-3.1}, \eqref{s-3.2} under a piecewise constant policy \(U^h\in\Uadm\) which approximates \(U\) weakly \(a.s\) (existence follows from the density results see \cref{p.w_denseness}). 
Then, for every $T>0$, there exists a constant $C_{_{EM}}>0$  independent of $h$, \(\gamma\in(0,\frac{1}{2})\), such that
\begin{equation}\label{s-3.3.1}
\Exp\!\left[\sup_{0\le t\le T}|X_t-X_t^h|^2\right]
\le C_{_{EM}}\,(h^{\gamma}+(\eta_{1,h})^{\gamma}),
\end{equation}
where $\eta_{1,h}\to0$ as $h\to0$. Consequently,

\begin{equation}\label{s-3.4}
\lim_{h\to0}\Exp\!\left[\sup_{0\le t\le T}|X_t-X_t^h|^2\right]=0.
\end{equation}
\end{theorem}
\begin{proof}
  Set $Z_t=X_t-X^h_t$ for $t\geq 0$, then $Z_0=X_0-X^h_0=0$  and
  \begin{equation*}
    Z_t=\int_0^t (b(X_s,S_s,U_s)-b(X^h_{s_h},S_{s_h}^h,U_{s_h}^h))\D s\;+\;\int_0^t (\upsigma(X_s,S_s)-\upsigma(X^h_{s_h},S_{s_h}^h))\D W_s,\ \ t>0.
  \end{equation*}
  Define first mode mismatch time \(\theta^h\df\inf\{t\ge0\;:\; S_t\neq S^h_t\}\). Since we have \(S_0=S^h_0\in\mathbb{S}\), \(\mathbb{P}(S_{\theta^h}\neq S^h_{\theta^h})>0\). Now we decompose the error according to whether the switching components remain coupled up to time $T$:
\begin{equation}\label{EMdecompo}
  \Exp\left[\sup_{0\leq t\leq T}|Z_t|^2
    \right]\leq \Exp\left[\sup_{0\leq t\leq T}|Z_t|^2\mathbf{1}_{\{\theta^h\le T\}}+\sup_{0\leq t\leq T}|Z_t|^2\mathbf{1}_{\{\theta^h>T\}}\right]
  \end{equation}
We estimate both terms separately; first, consider the second term. On $\{\theta^h>T\}$ the regime processes agree, so up to $T\wedge \theta^h$ we may compare only the diffusion components. We first do it for arbitrary \(t\in (0,T]\)
\begin{align}
    \Exp\left[\sup_{0\leq u\leq t}|Z_u|^2\mathbf{1}_{\{\theta^h>t\}}\right]
    &\leq \Exp\left[\sup_{0\leq u\leq t\wedge \theta^h}|Z_u|^2\right]\nonumber
\end{align}
  Using the Lipschitz continuity of $b$ and $\upsigma$ \hyperlink{B1}{(B1)} together with the BDG (
Burkholder–Davis–Gundy) inequality, we obtain
  \begin{align}\label{s-3.5}
    \Exp\left[\sup_{0\leq u\leq t\wedge \theta^h}|Z_u|^2\right]
    &\leq 2\Exp\bigg[\sup_{0\leq u\leq t\wedge \theta^h}\Big|\int_0^{u}b(X_s,S_s,U_s)-b(X^h_{s_h},S_{s_h}^h,U_{s_h}^h)\D s \Big|^2\bigg]\;\nonumber\\
    &\quad+\;2\Exp\bigg[\sup_{0\le u\le t\wedge \theta^h}\Big|\int_0^{u}\upsigma(X_s,S_s)-\upsigma(X^h_{s_h},S_{s_h}^h)\D W_s\Big|^2\bigg]\nonumber\\
    &\leq 8t\Exp\bigg[\int_0^{t\wedge \theta^h}\Big\{|b(X_s,S_s,U_s)-b(X^h_s,S_s,U_s)|^2+|b(X^h_s,S_s,U_s)-b(X^h_{s_h},S_s,U_s)|^2\nonumber\\
    &\quad+|b(X^h_{s_h},S_s,U_s)-b(X^h_{s_h},S^h_s,U_s)|^2 +\! |b(X^h_{s_h},S^h_s,U_s)\!-\!b(X^h_{s_h},S^h_{s_h},U_s)|^2\Big\}\D s\bigg]\!\nonumber\\
    &\quad+2\Exp\bigg[\sup_{0\leq u\leq t\wedge \theta^h}\Big|\int_0^{u}b(X^h_{s_h},S^h_{s_h},U_s)\!-\! b(X^h_{s_h},S_{s_h}^h,U^h_{s})\nonumber\\
    &\quad+b(X^h_{s_h},S^h_{s_h},U^h_s)\!-\! b(X^h_{s_h},S_{s_h}^h,U^h_{s_h})\Big\}\D s \Big|^2\bigg]\nonumber\\
& \quad+2C_{_{BDG}}\Exp\bigg[\int_0^{t\wedge \theta^h}\|\upsigma(X_s,S_s)-\!\upsigma(X^h_{s_h},S^h_{s_h})\|^2\D s\bigg] \nonumber\\
&\leq 8t\Exp\bigg[\int_0^{t\wedge \theta^h}\Big\{|b(X_s,S_s,U_s)-b(X^h_s,S_s,U_s)|^2+|b(X^h_s,S_s,U_s)-b(X^h_{s_h},S_s,U_s)|^2 \nonumber\\
    &\quad+|b(X^h_{s_h},S_s,U_s)-b(X^h_{s_h},S^h_s,U_s)|^2 +\! |b(X^h_{s_h},S^h_s,U_s)\!-\!b(X^h_{s_h},S^h_{s_h},U_s)|^2\Big\}\D s\bigg]\!+\eta_{1,h}\nonumber\\
    &\quad+8C_{_{BDG}}\Exp\bigg[\int_0^{t\wedge \theta^h}\Big\{|\upsigma(X_s,S_s)-\upsigma(X^h_s,S_s)|^2+|\upsigma(X^h_s,S_s)-\upsigma(X^h_{s_h},S_s)|^2\nonumber\\
    &\quad+|\upsigma(X^h_{s_h},S_s)-\upsigma(X^h_{s_h},S^h_s)|^2 +\! |\upsigma(X^h_{s_h},S^h_s)\!-\!\upsigma(X^h_{s_h},S^h_{s_h})|^2\Big\}\D s\bigg]\;\nonumber\\
&\leq\tilde{C}\Exp\bigg[\int_0^{t\wedge \theta^h}\Big\{ C_L\big(|Z_s|^2+|X^h_s-X^h_{s_h}|^2\big) + 2M_2\big(\mathbf 1_{\{S_s\neq S^h_s\}}+\mathbf 1_{\{S^h_s\neq S_{s_h}^h\}}\big)\Big\}\D s\bigg]\nonumber\\
&\qquad+\eta_{1,h} 
\end{align}
where \(C_{_{BDG}}\) is the BDG constant, \(\tilde{C}=8\max\{t,C_{_{BDG}}\}\) and  \[\eta_{1,h}=2\Exp\left[\sup_{0\le u\le t\wedge \theta^h}\Big|\int_0^{u} (b(X^h_{s_h},S^h_{s_h},U_s)\!-\! b(X^h_{s_h},S_{s_h}^h,U^h_{s}))\D s\Big|^2\right] \] Since \(S_s=S_s^h\) for \(s\le \theta^h\), the term \(\mathbf 1_{\{S_s\neq S^h_s\}}=0\) in \cref{s-3.5}.

\medskip
We next bound other terms of \cref{s-3.5} separately.
  \medskip
\noindent
  
\textbf{(i) Time discretization error.}
From the Euler–Maruyama scheme \eqref{s-3.1} and \hyperlink{B1}{(B1)}, we get
  \begin{align}\label{s-3.6}
    \Exp|X^h_s-X^h_{s_h}|^2&\leq  2h\Exp\int_{s_h}^s |b(X^h_{r_h}, S^h_{r_h},U^h_{r_h}))|^2\D r  + 2\Exp\int_{s_h}^s \|\upsigma(X^h_{r_h}, S^h_{r_h})\|^2\D r\nonumber\\
    &\leq 2 M_2h
  \end{align}
  
  \medskip
\noindent
\textbf{(ii) Numerical switching error.}
   For $t>0$, set $K=\lfloor\frac{t}{h}\rfloor$, $t_k=kh$ for $k\leq K$ and $t_{K+1}=t$.
 Then, according to \eqref{s-3.2b} and from boundedness of the jump-rates,
 
  \begin{equation}\label{s-3.7}
  \int_0^{t\wedge \theta^h}\Exp[\mathbf 1_{\{S^h_s\neq S_{s_h}^h\}}]\D s \le\sum_{k=0}^K\int_{t_k}^{t_{k+1}} \mathbb{P}(S^h_s\neq S^h_{t_k})\D s\leq Mh t +o(h).
  \end{equation}
  
  \medskip
\noindent
\textbf{(iii) True–numerical mismatch.}
 Substituting  \eqref{s-3.6}, \eqref{s-3.7} into \eqref{s-3.5}, we obtain
  \begin{align*}
  \Exp\big[\sup_{0\leq u\leq t\wedge \theta^h}|Z_u|^2\big]\leq \hat C(t)h+\eta_{1,h}+\tilde{C}C_L\int_0^t\Exp\big[\sup_{0\leq r\leq s\wedge \theta^h}|Z_r|^2\big] \D s
  \end{align*}
  where \(\hat C(t)=2\tilde{C}M_2t(1+M)+o(h)\). By the Grönwall’s inequality, we obtain that
  \begin{equation}\label{EMdecomp2}
   \Exp\left[\sup_{0\leq u\leq t\wedge \theta^h}|Z_u|^2\right]\leq (\hat C(t)h+\eta_{1,h}) e^{C_L\tilde{C}t}
  \end{equation}
Since \cref{EMdecomp2} holds for any \(t\in(0,T]\), it holds for \(t=T\).

Now we estimate the first term of \cref{EMdecompo}. To this end, we first estimate 
$\mathbb{P}\{\theta^h \leq T\}$. From \cref{E1.1,s-3.2}, we have
\begin{equation}
S_t - S_t^h
= \int_0^t \int_{\RR_+}
\bigl[
h(X_{s}, S_{s^-}, z)
- h(X^h_{s_h}, S^h_{s^-}, z)
\bigr] \, \cP(ds,dz).
\end{equation}
Note that $S_t^h=S_t$ for all $t<\theta^h$, using \cref{t-3.1} 
and \hyperlink{B1}{(B1)} in the first two inequalities below to compute
\begin{align}\label{mismatch_bound}
\mathbb{P}\{\theta^h \leq T\}
&= \Exp\bigl[ \mathbf{1}_{\{S_{T \wedge \theta^h}\neq S^h_{T \wedge \theta^h}\}} \bigr] \nonumber \\
&= \Exp\bigg[ \int_0^{T \wedge \theta^h} \int_{\RR_+}
\mathbf{1}_{\{ h(X_s,S_{s^-},z)\neq  h(X^h_{s_h},S^h_{s^-},z)\}}
\, \cP(ds,dz)\bigg] \nonumber \\
&= \Exp\bigg[ \int_0^{T \wedge \theta^h} \int_{\RR_+}
\mathbf{1}_{\{ h(X_s,S_{s^-},z) \neq h(X^h_{s_h},S^h_{s^-},z)\}}
\, \mathbf m(dz)\, ds \bigg]\nonumber \\
&\le \Exp\bigg[ \int_0^{T \wedge \theta^h}
\sum_{l \neq S_{s^-}}
\bigl| m_{S_{s^-}l}(X_s) - m_{S_{s^-}l}(X^h_{s_h}) \bigr| ds\bigg] \nonumber \\
&\le \Exp\bigg[ \int_0^{T \wedge \theta^h}
|\mathbb{S}|C_L | X_s-X^h_{s_h}| \, ds\bigg] \nonumber \\
&\le |\mathbb{S}|C_L \Exp\bigg[ \int_0^{T \wedge \theta^h}
\bigl( |X^h_s - X_s| + |X^h_{s_h} - X^h_s| \bigr) ds\bigg]   \nonumber \\
&\le |\mathbb{S}|C_L \int_0^T
\left(
\Exp\Big[\sup_{0 \le u \le s \wedge \theta^h}
|X_u - X^h_u|^2\Big]^{1/2}
+ \Exp|X^h_s-X^h_{s_h} |
\right) ds  \nonumber \\
&\le |\mathbb{S}|C_L \int_0^T \left( (\hat C(T)h+\eta_{1,h}) e^{C_L\tilde{C}T})^{1/2} + 2M_2h \right) ds   \nonumber \\
&\le  \hat{C}_{EM}( 2h^{1/2}+(\eta_{1,h})^{1/2}),
\end{align}
where \( \hat{C}_{EM}=\mathbb{S}C_LT\max\{\hat C(T)e^{C_L\tilde{C}T\frac 12},2M_2,e^{C_L\tilde{C}T\frac 12}\}\) 
and the second last inequality follows from \cref{EMdecomp2,s-3.6}.

Then, we can use Hölder's inequality and \hyperlink{B1}{(B1)} to derive
\begin{align}\label{EMdecompo1}
\Exp\Big[
\mathbf{1}_{\{\theta^h \le T\}}
\sup_{0 \le t \le T}
| X_t-X^h_t |^2
\Big]
&\le \bigl( \mathbb{P}\{\theta^h \le T\} \bigr)^{1/p}
\left(
\Exp\bigg[ \sup_{0 \le t \le T}
| X_t-X^h_t |^{2q}\bigg]
\right)^{1/q} \nonumber \\
&\le 2M_2T \hat{C}_{EM}^{\frac{1}{p}}( 2h^{1/2}+(\eta_{1,h})^{1/2})^{\frac{1}{p}},
\end{align}
where \(p,q>1\) such that \(\frac{1}{p}+\frac{1}{q}=1\).

Finally by substituting \cref{EMdecomp2,EMdecompo1} into \cref{EMdecompo}, we obtain
\begin{align*}
    \Exp\Big[
\sup_{0 \le t \le T}
| X_t-X^h_t |^2
\Big]&\le2M_2T  \hat{C}_{EM}^{\frac{1}{p}}( 2h^{1/2}+(\eta_{1,h})^{1/2})^{\frac{1}{p}}+(\hat C(T)h+\eta_{1,h}) e^{C_L\tilde{C}T}\\
&\le C_{EM}(h^{\frac{1}{2p}}+(\eta_{1,h})^{\frac{1}{2p}})
\end{align*}
 for some constant $ {C}_{EM}>0$ depending on $\hat{C}_{EM}$. The last inequality follows from the sub-additivity of the concave function \(f(x)=x^{\frac{1}{p}},\, p>1\).  In particular for any \(\gamma\in (0,\frac{1}{2})\) take \(p=\frac{1}{2\gamma}\) to obtain \cref{s-3.3.1}.
Since \(U^h_{(\cdot)}\to U\) weakly $a.s$ in $\Omega$, arguing as in the proof of \cite{pradhan2025nearoptimalitydiscretetimeapproximations}*{Theorem 3.1}  implies \(\eta_{1,h}\to 0\) as \(h\to 0\) which  yields the desired conclusion.
\end{proof}
\begin{corollary}[Rate under aligned piecewise-constant controls]\label{corollary-em}
Under the assumptions of \cref{t-3.3}, if the same piecewise-constant control is used for both the true system and the Euler-Maruyama scheme, then for any \(\gamma\in (0,\frac{1}{2})\),
\begin{equation}\label{s-3.3}
\Exp\left[\sup_{0\leq t\leq T}|X_t-X^h_t|^2\right] \leq C_{_{EM}}h^{\gamma}
\end{equation} for some constant $C_{_{EM}}>0$ depending on $T$ and independent of $h$.
\end{corollary}

\subsubsection{\bf Convergence for the finite horizon cost criterion}\hspace{30em}
\vspace{.5em}

We now analyze the Euler–Maruyama time discretization of the controlled
switching diffusion and establish convergence of the corresponding value
functions.
Under the discrete-time setup, the associated discrete-time cost evaluation criteria are given by:\\ \\
\textbf{Discrete-time Finite Horizon Cost:}

Fix a time step $h>0$ and define the grid $t_k = kh$ with
$K=\lfloor T/h \rfloor$ and $t_K \le T < t_{K+1}$. Define the stage-wise cost function \(c_h:\Rd\times\mathbb{S}\times\Act\to\RR_{+}\) such that for every \((x,i,\zeta)\in \Rd\times\mathbb{S}\times\Act\)
\begin{equation}\label{discrete-cost}
c_h(x,i,\zeta)\df h\times c(x,i,\zeta).
\end{equation}
where $c$ is the cost function of the switching diffusion model.
For a control sequence $U^h=(U_{t_k}^h)_{k\ge0}$, the discrete-time
finite-horizon cost is defined by
\begin{equation}\label{Ediscrete_cost}
\cJ_{T,h}^{U^h}(x,i)
=
\mathbb{E}_{x,i}^{U^h}
\!\left[
\sum_{k=0}^{K-1}
c_h(X_{t_k}^h,S_{t_k}^h,U_{t_k}^h)
+
c_T(X_{t_K}^h,S_{t_K}^h)
\right].
\end{equation}

An \emph{admissible policy} is a sequence
$\{U_{t_k}^h\}_{k\ge0}$ such that each $U_{t_k}^h$ is measurable with
respect to the information set
\[
I_k^h
=
\{X_{[0,t_k]}^h,\, S_{[0,t_k]}^h,\, U_{[0,t_{k-1}]}^h\},
\qquad
I_0^h=\{X_0^h,S_0^h\},
\]
that is,
\begin{equation}\label{Econtrol}
U_{t_k}^h = v_k^h(I_k^h),
\end{equation}
for some measurable map $v_k^h$ taking values in $\pV$.
The collection of all such controls is denoted by $\Uadm^h$.

We further introduce the following subclasses:

\begin{itemize}
\item
$\Um^h$ (Markov policies):
\[
U_{t_k}^h = v_k^h(t_k,X_{t_k}^h,S_{t_k}^h)
\quad\text{for measurable } v_k^h:[0,\infty)\times\Rd\times\mathbb{S}\to \pV.
\]

\item
$\Usm^h$ (stationary Markov policies):
\[
U_{t_k}^h = v^h(X_{t_k}^h,S_{t_k}^h)
\quad\text{for a time-independent measurable map } v^h:\Rd\times\mathbb{S}\to \pV. 
\]
\end{itemize}
The discrete-time control problem is to find a control sequence $U^{h, *}$ that minimizes $\cJ_{T,h}^{U^h}(x)$: that is 
\begin{equation} \label{Ediscrete_optimal_control}
V_T^h(x) \df \inf_{U^h \in \Uadm^h} \cJ_{T,h}^{U^h}(x) = \cJ_{T,h}^{U^{h, *}}(x).
\end{equation}
Here $V_T^h(x)$ is the discrete-time value function.

Let us denote
\begin{equation}\label{kernel}
P_h(dy,j\mid x,i,\zeta)=\mathbb{P}_{x,i}((X^h_{t_{k+1}},S^h_{t_{k+1}})\in dy\times\{j\}\,|\,X^h_{t_k}=x,S^h_{t_k}=i,U^h_{t_k}=\zeta)
\end{equation}
 be the transition kernel of the Markov chain \cref{s-3.1,s-3.2}.

\textbf{Note:} For any \(f\in \cC_b(\Rd\times\mathbb{S})\), we will adopt the following convention
\[
\int_{\Rd\times\mathbb{S}} f(y,j)\,P_h(dy,dj \mid x,i,\zeta)
=
\sum_{j\in\mathbb{S}} \int_{\Rd} f(y,j)\,P_h(dy,j \mid x,i,\zeta).
\]

From \cite{hernandez2012discrete}*{Theorem 3.2.1} we have the following verification theorem for the finite horizon cost in our setup.

\begin{theorem}[Verification Theorem for Finite Horizon Cost]\label{thm:verification_finite} Suppose the assumptions \hyperlink{B1}{(B1)}-\hyperlink{B2}{(B2)} hold for a Markov decision process with the finite horizon cost defined in \cref{Ediscrete_cost}.
Let $\{V_{t_k}(x,i)\}_{k=0}^{K}$ satisfies the dynamic programming equations
\begin{align}
V_{t_k}(x,i)
&=
\min_{\zeta\in \Act}
\Bigg[
c_h(x,i,\zeta)
+
\int_{\Rd\times\mathbb{S}}V_{t_{k+1}}(y,j)
P_h(dy,dj|(x,i),\zeta)\,
\Bigg],\\
V_K(x,i) &= c_{T}(x,i),
\end{align}
for $k=K-1,\dots,0$.

Then the following statements hold.
\begin{enumerate}
\item The function $V_0(x,i)$ coincides with the optimal value function, i.e.,
\[
V_0(x,i)=\inf_{U^h \in \Uadm^h} \cJ_{T,h}^{U^h}(x,i) = \cJ_{T,h}^{U^{h, *}}(x,i)
\]
\item A Markov policy $v_k$ is optimal if and only if it attains the minimum in the dynamic programming at every state and time, i.e.,
\begin{align*}
c_h(x,i,v_k(k,x,i))&+
\int_{\Rd\times\mathbb{S}}V_{t_{k+1}}(y,j)P_h(dy,dj|(x,i),v_k(k,x,i))\\
&=\min_{\zeta\in \Act}
\Big[
c_h(x,i,\zeta)+
\int_{\Rd\times\mathbb{S}}V_{t_{k+1}}(y,j)P_h(dy,dj|(x,i),\zeta)
\Big],
\end{align*}
\end{enumerate}
\end{theorem}

\begin{remark}\label{bdd&cts}
The value functions $\{V_{t_k}\}_{k=0}^{K-1}$ are bounded and continuous on $\Rd\times\mathbb{S}$, which follows by backward induction from the verification theorem and the weak continuity of the transition kernel.
\end{remark}

\begin{theorem}[Convergence of discrete-time value functions]\label{Tconvevaluefunc}
Suppose Assumptions \hyperlink{B1}{(B1)}-\hyperlink{B2}{(B2)} hold. Let $V_T^h(x,i)$ be the value function in the discrete-time model (corresponding to the piecewise constant policy $v^{*\epsilon}$ (see \cref{nearop-finite horizon})). Then, for all $(x,i)\in \Rd\times\mathbb{S}$ and \(\gamma\in (0,\frac{1}{2})\) we have, there exists a positive constant $\hat{C}_5$ such that
\begin{equation*}
\abs{V_T^h(x,i) - V_T(x,i)} \leq  \hat{C}_5h^{\frac \gamma 2}\,.
\end{equation*}
\end{theorem}

\begin{proof}
We first compare the discrete sum with the continuous-time integral.
For any admissible $U^h\in \Uadm^h$, let \((X_{t_k}^h,S_{t_k}^h)\) be the discrete-time controlled process obtained from \cref{s-3.1}-\cref{s-3.2}. 
\begin{align*}
& \Bigg | \mathbb{E}_{x,i}^{U^h} \left[ \sum_{k=0}^{ K -1 } c(X_{t_k}^h, S_{t_k}^h, U_{t_k}^h)h + c_T(X^h_{t_K}, S_{t_K}^h))\right]    - \mathbb{E}_{x,i}^{U^h} \left[ \int_{0}^{T} c(X^h_s, S^h_s, U^h_s) ds + c_T(X^h_T, S^h_T)\right]\Bigg | \nonumber\\ 
&\quad\leq  \Bigg |\mathbb{E}_{x,i}^{U^h} \left[\sum_{k=0}^{ K-1  } \int_{t_k}^{t_{k+1}} \left(c(X_{t_k}^h, S_{t_k}^h, U_{t_k}^h) - c(X^h_s, S^h_s, U^h_s)\right) ds \right]\Bigg | + \Bigg |\mathbb{E}_{x,i}^{U^h} \left[ \int_{t_K}^{T} c(X^h_s,S^h_s, U^h_s) ds \right]\Bigg |\nonumber\\
&\quad=\Bigg |\mathbb{E}_{x,i}^{U^h} \left[ \int_{t_K}^{T} c(X^h_s,S^h_s, U^h_s) ds \right]\Bigg |\nonumber\\
&\quad\leq  \|c\|_{\infty}h
\end{align*}

Thus, it follows that
\begin{equation*}
\mathcal{J}_{T,h}^{U^h}(x,i) = \mathbb{E}_{x,i}^{U^h} \left[ \int_{0}^{T}  c(X^h_s,S^h_s,U^h_s)\D s + c_T(X^h_T,S^h_T)\right] + \order(h).
\end{equation*}
Let $v^{*\epsilon}$ is a Lipschitz continuous $\epsilon$-optimal control for the continuous time problem (existence is guaranteed by \cref{nearop-finite horizon}) then, we have 
\begin{align}\label{bound22}
\cJ_T^{v^{*\epsilon}}(x,i) \leq V_T(x,i) + \epsilon.
\end{align}
Now we define a piecewise constant control
\begin{equation}\label{p.w-policy}
    v^{*\epsilon h}(s,X_s^h,S_s^h)\df 
v^{*\epsilon}(kh,X_{kh}^h,S_{kh}^h)\qquad \text{for}\, s\in[kh,(k+1)h)
\end{equation}
 Consider the discrete-time model $(X^h,S^h)$ associated with the piecewise constant control $v^{*\epsilon h}$ (as in \cref{s-3.1}-\cref{s-3.2}); with optimal value $V_T^h(x,i)$. Then, it follows that 
 \begin{align}\label{bound23}
 V_T^h(x,i) \leq \cJ_{T,h}^{v^{*\epsilon h}}(x,i).
 \end{align}
Since $c\; \text{and}\; c_T$ are Lipschitz continuous, for some constant $C_5 > 0$ we get
\begin{align*}
\abs{\cJ_T^{v^{*\epsilon}}(x,i) - \cJ_{T,h}^{v^{*\epsilon h}}(x,i)} \leq & \abs{\Exp_{x,i} \left[ \int_0^T c(X_t, S_t, v^{*\epsilon}_t) dt + c_T(X_T, S_T) \right] \nonumber\\
&- \Exp_{x,i} \left[ \int_0^T c(X^h_t, S^h_t, v^{*\epsilon h}_t) dt + c_T(X^h_T, S^h_T) \right]} + \abs{\order(h)}\nonumber\\
\leq & \Exp_{x,i} \left[ \int_0^T \abs{(c(X_t, S_t, v^{*\epsilon}_t)-c(X^h_t, S^h_t, v^{*\epsilon h}_t))} dt  \right]+ \abs{\order(h)} 
\nonumber \\
\leq & \Exp_{x,i} \Big[ \int_0^T \big\{\abs{(c(X_t, S_t, v^{*\epsilon}_t)-c(X^h_t, S_t, v^{*\epsilon }_t))}\\
&\,+\abs{(c(X^h_t, S_t, v^{*\epsilon}_t)-c(X^h_t, S^h_t, v^{*\epsilon}_t))}(\mathbf{1}_{\{\theta^h\le t\}}+\mathbf{1}_{\{\theta^h> t\}})\\
&\,+\abs{(c(X^h_t, S^h_t, v^{*\epsilon}_t)-c(X^h_t, S^h_t, v^{*\epsilon h}_t))}\big\} dt  \Big]+ \abs{\order(h)} \nonumber \\
\leq & C_5\Exp_{x,i}\int_0^T\abs{X_t - X^h_t}\D t + 4M_3Ch^{\frac 12}+\eta_{3,h}+ \|c\|_{\infty} h
\nonumber \\
\leq & C_5T\sqrt{\Exp_{x,i}[\sup_{t\leq T}\abs{X_t - X^h_t}^2]}+ 4M_3Ch^{\frac 12} +C_{L_c}C_2(\epsilon)hT+ \|c\|_{\infty} h
\end{align*}
where, 
\begin{align*}
    \eta_{3,h}&=\Exp_{x,i} \int_0^T\abs{(c(X^h_t, S^h_t, v^{*\epsilon}_t)-c(X^h_t, S^h_t, v^{*\epsilon h}_t))}dt\\
&= \Exp_{x,i} \sum_{k=0}^K\int_{t_k}^{t_{k+1}}\abs{(c(X^h_t, S^h_t, v^{*\epsilon}_t)-c(X^h_t, S^h_t, v^{*\epsilon h}_t))}dt\\
&\le C_{L_c}C_2(\epsilon)hT\qquad(\text{from}\, \hyperlink{B2}{(B2)} \,\text{and}\,\cref{p.w-policy}  )
\end{align*}
where the fourth inequality follows from the \cref{mismatch_bound} and \hyperlink{B2}{(B2)} (since in this case \(\eta_{1,h}, \eta_{2,h}\) are zero). By \cref{corollary-em},
$\mathbb{E}[\sup_{t\leq T}\abs{X_t - X^h_t}^2] \leq C_{_{EM}} h^{\gamma}$ for \(\gamma\in(0,\frac{1}{2})\). Hence
\begin{align}\label{ETconvevaluefunc1A}
\abs{\cJ_T^{v^{*\epsilon}}(x,i) - \cJ_{T,h}^{v^{*\epsilon h}}(x,i)} \leq (C_5C_{_{EM}}^{\frac 1 2} +4M_3C+ C_{L_c}C_2(\epsilon)T+\|c\|_{\infty})h^{\frac \gamma 2} = \hat{C}_5h^{\frac \gamma 2}, 
\end{align}
where the constant $\hat{C}_5 := (C_5C_{_{EM}}^{\frac 1 2} +4M_3 C+C_{L_c}C_2(\epsilon)T+\|c\|_{\infty})$ depends upon the piecewise constant policy $v^{*\epsilon h}$\,.
Thus, from \eqref{bound22}, \eqref{bound23} and \eqref{ETconvevaluefunc1A} we get
\begin{equation}\label{EconvevaluefuncA}
V_T^h(x,i) \leq \cJ_{T,h}^{v^{*\epsilon}}(x,i) \leq \cJ_{T}^{v^{*\epsilon}}(x,i) + \hat{C}_5h^{\frac \gamma 2} \leq V_T(x,i) + \epsilon +\hat{C}_5h^{\frac \gamma 2}\,.
\end{equation}
For the lower bound, let $U^{h,*}\in \Uadm^h$ be an optimal control of the discretized system\,. Then, in view of (\ref{ETconvevaluefunc1A}), it follows that
\begin{equation}\label{EconvevaluefuncB}
V_T^h(x,i) = \cJ_{T,h}^{U^{h, *}}(x,i) \geq \cJ_{T}^{U^{h,*}}(x,i) - \hat{C}_5h^{\frac \gamma 2} \geq  V_T(x,i) - \hat{C}_5h^{\frac \gamma 2}\,.
\end{equation} 
Since $\epsilon$ is arbitrary, from (\ref{EconvevaluefuncA}) and (\ref{EconvevaluefuncB}) we obtain the desired result\,. This completes the proof\,.
\end{proof}
\begin{remark}
The constant $\hat{C}_5 = \hat{C}_5(\epsilon)$ depends on $\epsilon$. 
However, as the discretization step $h$ is uniform, this dependence does not influence the convergence result.
\end{remark}

\begin{corollary}\label{cor:CNearFinite}
From the proof of the above theorem (as we have obtained the estimate \cref{ETconvevaluefunc1A}), it follows that for piecewise-constant policies (aligned with the Euler-Maruyama grid) there is a uniform error estimate for \(\gamma\in (0,\frac{1}{2})\),
\[
\sup_{v^{h}\ \text{piecewise-const on grid }h}\big|\cJ_T^{v^{h}}(x,i) - \cJ_{T,h}^{v^{h}}(x,i)\big| \le \tilde{C}\,h^{\frac{^{\gamma}}{2}},
\]
for some constant $\tilde{C}$ that depends only on the model parameters and $T$.
\end{corollary}
Next, using the above continuity result of the value functions with respect to the discrete-time approximation, we establish the near optimality of the discrete-time optimal policy in the continuous-time RSDP system.
\begin{theorem}[Near optimality of discrete-time optimal policies]\label{TNearFinite}
Assume \hyperlink{B1}{(B1)}--\hyperlink{B2}{(B2)} holds.
Let $U^{*,h}\in\Um^h$ be an optimal policy for the discrete-time approximation, and let $\tilde U^{*,h}$ denote its continuous-time interpolation.

Then there exists a constant $\hat C_6>0$, independent of $h$, such that for any \(\gamma\in (0,\frac{1}{2})\) 
\[
\big|\cJ_T^{\tilde U^{*,h}}(x,i)-V_T(x,i)\big|
\le \hat C_6\, h^{\gamma/2}.
\]
\end{theorem}
\begin{proof}
By the triangle inequality for each $(x,i)\in\Rd\times\mathbb{S}$, (since $V_T^{h}(x,i) = \cJ_{T,h}^{U^{*, h}}(x,i)$) we have
\begin{equation}
|\cJ_{T}^{\tilde{U}^{*, h}}(x,i) - V_T(x,i)| \leq |\cJ_{T}^{\tilde{U}^{*, h}}(x,i) - \cJ_{T,h}^{U^{*, h}}(x,i)| + |V_T^{h}(x,i) - V_T(x,i)|
\end{equation}
each term is $O(h^{\gamma/2})$ by \cref{Tconvevaluefunc} and
\cref{cor:CNearFinite}.
\end{proof}


\subsection{Finite State Approximation of Controlled Regime-Switching Diffusions}

Consider the Markov chain defined in \cref{s-3.1,s-3.2}, taking values in
$\Rd \times \mathbb{S}$, with transition kernel
$P_h(dy,j \mid x,i,\zeta)$ and one-stage cost $c_h$ (see \cref{kernel,discrete-cost}).
We denote the corresponding Markov decision process (MDP) by
\(
\mathcal{M}_h = (\Rd \times \mathbb{S}, \Act, P_h, c_h).
\)
To approximate $\mathcal{M}_h$, we construct a sequence of finite-state models
$\widehat{\mathcal{M}}_{h,n}$. To this end, we first define a compact state MDP.

For any compact set $\cK \subset \Rd$, 
we define the compact state MDP
\(
\mathcal{M}^{\cK} := (\cK \times \mathbb{S}, \Act, p_{\cK}, c_{\cK}),
\)
where:
\begin{itemize}
    \item $p_{\cK}$  
    is a weakly continuous transition kernel in $x,\zeta$, i.e., for $(x_n,\zeta_n)\to (x,\zeta)$, $p_{\cK}(\cdot,\cdot\mid x_n,i,\zeta_n)\to p_{\cK}(\cdot,\cdot\mid x,i,\zeta)$ weakly as \(n\to\infty\) for $i\in\mathbb{S}$.
    \medskip
    \item $c_{\cK}:\cK\times\mathbb{S}\times\Act\to [0,\infty)$ is a continuous one-stage cost function.
\end{itemize}

\medskip
We first approximate the compact-state MDP $\mathcal{M}^{\cK}$ by
finite-state models, which will serve as intermediate approximations
to the original model $\mathcal{M}_h$.

\subsubsection{Finite State Approximation of $\mathcal{M}^{\cK}$}
We first construct a sequence of finite states for a given compact space \(\cK\subset\Rd\) by similar construction as in \cref{finite-action-construction}, let \(\Lambda_n=\{z_{n,l}\}_{l=1}^{l_n}\) be the sequence of finite grids constructed as in \cref{finite-action-construction} for the space $\cK$ with the nearest neighbour quantizer \(\mathcal{Q}_n:\cK\to\Lambda_n\) such that \[\lim_{n\to\infty}
\sup_{z\in \cK} |z-\mathcal{Q}_n(z)| \to 0.
\]

For each $n$, let $\{\mathcal{S}_{n,l}\}_{l=1}^{l_n}$ be the partition of $\cK$ induced by $\mathcal{Q}_n$ and is given by
\begin{align}
\mathcal{S}_{n,l} = \{z \in \cK: \mathcal{Q}_n(z)=z_{n,l}\}, \nonumber
\end{align}
with diameter $\diam(\mathcal{S}_{n,l}) \coloneqq \sup_{z,y\in\mathcal{S}_{n,l}} d_{\cK}(z,y) < 2/n$. Let $\{\nu_n\}$ be a sequence of
probability measures on $\cK$ satisfying
\begin{align}
\nu_n(\mathcal{S}_{n,l}) > 0 \text{  for all  } l,n.  \label{compact:numeas}
\end{align}
We let $\nu_{n,l}$ be the restriction of $\nu_n$ to $\mathcal{S}_{n,l}$ defined by
\begin{align}
\nu_{n,l}(\,\cdot\,) \coloneqq \frac{\nu_n(\,\cdot\,)}{\nu_n(\mathcal{S}_{n,l})}. \nonumber
\end{align}
The measures $\nu_{n,l}$ will be used to define a sequence of finite-state MDPs, denoted as $\widehat {\mathcal{M}}_{h,n}$ ($n\geq1$), to approximate the original model $\mathcal{M}_h$.

\medskip
\noindent
Now we construct a finite-state approximation for $\mathcal{M}^{\cK}$. We approximate the model \(\mathcal{M}^{\cK}\) by the finite state models \(\widehat{\mathcal{M}}_{n}^{\cK}\), defined as follows:

\begin{enumerate}
    \item The state space for the process is $\Lambda_n $
    \item \medskip

\noindent
The transition kernel $p_{n}: \Lambda_n\times\mathbb{S}\times\Act \to \sP(\Lambda_n\times\mathbb{S})$ is defined by
\begin{align*}
p_{n}(\cdot,j\mid z_{n,l},i,\zeta)
=\int_{\mathcal{S}_{n,l}} p_{\cK}\big(\mathcal{Q}_{n}^{-1}(\cdot)\cap\cK,\,j \mid z,i,\zeta\big)\nu_{n,l}(dz),
\end{align*}

\medskip

\noindent
The one-stage cost function $c_{n}:\Lambda_n\times\mathbb{S}\times\Act\to[0,\infty)$ is defined as
\[
c_n(z_{n,l},i,\zeta)=
\int_{\mathcal{S}_{n,l}} c_{\cK}(z,i,\zeta)\nu_{n,l}(dz),   
\]
with the terminal cost \[c_T^n(z_{n,l},i)=c_T(z_{n,l},i)\qquad\text{for}\,\, (z_{n,l},i)\in\Lambda_n\times\bS.\] 
\end{enumerate}


\subsubsection{Finite-horizon cost approximation}

Here we show that the finite-horizon cost defined for the finite state MDPs \(\widehat{\mathcal{M}}_{n}^{\cK}\) 
converges to the finite-horizon cost of the compact state MDP 
$\mathcal{M}^{\cK}$.

\medskip
\noindent
\textbf{Finite-state finite-horizon cost function.}

For each $n$, \((x,i)\in \Lambda_n\times\mathbb{S}\), $k\in\{0,\ldots,K-1\}$, define a function
\begin{equation}\label{quantized_cost1}
\hat {\cJ}_k^n(x,i)=
\min_{\zeta\in \Act}
\left[
c_n(x,i,\zeta)+
\sum_{(y,j)\in \Lambda_n\times\mathbb{S}}
\hat \cJ_{k+1}^n(y,j)p_{n}(y,j\,\mid x,i,\zeta)
\right]
\end{equation}
Under assumptions \hyperlink{B1}{(B1)}-\hyperlink{B2}{(B2)}, \cite{hernandez2012discrete}*{Theorem 3.2.1} implies that $\hat {\cJ}_k^n$ is the value function for the MDP \(\widehat{\mathcal{M}}_{n}^{\cK}\) and there exists an optimal Markov policy as a measurable minimizing selector of  \cref{quantized_cost1}. 

\medskip
\noindent

We extend the functions  \(c_n\,\text{and}\, \hat {\cJ}_k^n\) on \(\cK\times\mathbb{S}\) (still denoted by \(c_n\,\text{and}\,\hat {\cJ}_k^n\) respectively) as
\begin{align*}
c_n(x,i,\zeta)&\df
c_{\cK}(\mathcal{Q}_{n}(x),i,\zeta),\\
  \hat {\cJ}_k^n(x,i)&\df \hat {\cJ}_k^n(\mathcal{Q}_{n}(x),i)\qquad 
 \end{align*}
 \(\text{for all}\, (x,i,\zeta)\in\cK\times\mathbb{S}\times\Act.\)
 
Therefore \(c_{n} \,\text{and}\, \hat {\cJ}_k^n\) are constant on the sets 
\(\mathcal{Q}_{n}^{-1}(x)\), for \(x\in \Lambda_n\). In particular, we can write the summation term in \cref{quantized_cost1} as 
\begin{align*}
    \sum_{(y,j)\in \cK\times\mathbb{S}}
\hat \cJ_{k+1}^n(y,j)p_{n}(y,j\,\mid x,i,\zeta)&=\int_{ \mathcal{S}_{n,h_n(x)}}\bigg( \int_{\cK\times\mathbb{S}}
\hat \cJ_{k+1}^n(y,j)p_{\cK}(dy,dj\;|z,i,\zeta)\bigg) \nu_{n,h_n(x)}(dz)
\end{align*}
where \(h_n:\cK\to \{1,\ldots,l_n\} \) maps \(x\) to the index of the partition \(\mathcal{S}_{n,l}\) it belongs to.

Define a function $\bar{\cJ}_{k}$ on $\cK\times\bS$ by
\begin{align}\label{cpt-cost}
\bar{\cJ}_{k}(x,i)=
\min_{\zeta\in \Act}
\left[
c_\cK(x,i,\zeta)+
\int_{\cK\times\mathbb{S}}
\bar{\cJ}_{k+1}(y,j)p_\cK(y,j\,\mid x,i,\zeta)\right]
\end{align}
From \cite{hernandez2012discrete}*{Theorem 3.2.1} the finite-horizon value function for the MDP $\mathcal{M}^{\cK}$ is $\bar{\cJ}_{k}$ and there exists an optimal Markov policy which is a pointwise minimizer of \cref{cpt-cost}.

\medskip
\noindent
Then we have the following convergence result on finite grids: the value function of the finite state model $\widehat{\mathcal{M}}^{\cK}_n$ converges to the value function of the compact MDP $\mathcal{M}^{\cK}$.
\begin{lemma}\label{compact-space_convergence}
Under the assumptions \hyperlink{B1}{(B1)}-\hyperlink{B2}{(B2)}, for all $k\in \{0,\ldots,K\}$, 
\[
\max_{(x,i)\in\Lambda_n\times\mathbb{S}}|\bar{\cJ}_k(x,i)-\hat {\cJ}_k^n(x,i)|
\to 0\qquad \text{as}\;\;n\to \infty.
\]
\end{lemma}
\begin{proof}

We proceed by backward induction.

At $k=K$, $\bar{\cJ}_K(x,i)=\hat{\cJ}_K^n(x,i)=c_T(x,i)$ for all \((x,i)\in
\Lambda_n\times\mathbb{S}\), hence the result holds.

Assume the result holds for $k+1$. Then for $(x,i)\in \Lambda_n\times\mathbb{S}$,
\begin{align*}
    |\bar{\cJ}_k(x,i)-\hat {\cJ}_k^n(x,i)|&\le \Bigg|\max_{\zeta\in \Act} \Bigg[ \int_{\cK\times\mathbb{S}} \bar{\cJ}_{k+1}(y,j)p_{\cK}(dy,dj\mid x,i,\zeta)-\sum_{(y,j)\in \Lambda_n\times\mathbb{S}}
\hat \cJ_{k+1}^n(y,j)p_{n}(y,j|x,i,\zeta)\Bigg]\Bigg|\\
  & \le \Bigg|\max_{\zeta\in \Act}\Bigg[\int_{\cK\times\mathbb{S}} \bar{\cJ}_{k+1}(y,j)p_{\cK}(dy,dj\mid x,i,\zeta)\\
  &\qquad-\int_{\mathcal{S}_{n,h_n(x)}} \Bigg(\int_{\cK\times\mathbb{S}}
\hat \cJ_{k+1}^n(y,j)p_{\cK}(dy,dj\;|z,i,\zeta)\bigg)\nu_{n,h_n(x)}(dz)\Bigg]\Bigg| \\
  & \le \Bigg|\max_{\zeta\in \Act}\Bigg[\int_{\cK\times\mathbb{S}} \bar{\cJ}_{k+1}(y,j)p_{\cK}(dy,dj\mid x,i,\zeta)\\
  &\qquad-\int_{\cK\times\mathbb{S}} 
\hat \cJ_{k+1}^n(y,j)\int_{\mathcal{S}_{n,h_n(x)}}p_{\cK}(dy,dj\;|z,i,\zeta)\nu_{n,h_n(x)}(dz)\Bigg]\Bigg|\\
  & \le \Bigg|\max_{\zeta\in \Act}\Bigg[\sup_{z\in\mathcal{S}_{n,h_n(x)}}\Bigg|\int_{\cK\times\mathbb{S}} \bar{\cJ}_{k+1}(y,j)p_{\cK}(dy,dj\mid x,i,\zeta)\\
  &\qquad-\int_{\cK\times\mathbb{S}} 
\hat \cJ_{k+1}^n(y,j)p_{\cK}(dy,dj\;|z,i,\zeta)\Bigg|\Bigg]\Bigg|
\end{align*}
By the induction hypothesis, together with the continuity and boundedness of 
$\hat{\cJ}_{k+1}^n$ (see \cref{bdd&cts}) and the weak continuity of $p_{\cK}$, 
we obtain, as $n \to \infty$, that 
\( 
\bigl|\bar{\cJ}_k(x,i) - \hat{\cJ}_k^n(x,i)\bigr| \to 0,\) \(\text{for all } (x,i)\in \Lambda_n\times\mathbb{S}, \;\; k=K-1,\ldots,0\).
Since \(\cK\) is compact, uniform continuity of costs implies  uniform convergence on \(\Lambda_n\times\mathbb{S}\).
\end{proof}

\medskip
\noindent
We extend this result to the compact subset $\cK$.

\begin{theorem}\label{finite-state-cost_approximation}
Suppose that Assumptions \hyperlink{B1}{(B1)}--\hyperlink{B2}{(B2)} hold. Then for each $k\in\{0,\ldots,K\}$
\[
\sup_{(x,i)\in \cK\times \mathbb{S}}
|\bar{\cJ}_k(x,i)-\hat {\cJ}_k^n(x,i)|
\to 0\qquad \text{as}\,\,n\to\infty.
\]
\end{theorem}
\begin{proof}
For \((x,i)\in\cK\times\mathbb{S}\), $k\in\{0,\ldots,K\}$
\begin{equation*}
    |\bar \cJ_{k}(x,i)-\hat \cJ_{k}^n(x,i)|
    \le |\bar\cJ_{k}(x,i)- \bar\cJ_{k}(\mathcal{Q}_{n}(x),i)|+|\bar\cJ_{k}(\mathcal{Q}_{n}(x),i)-\hat \cJ_{k}^n(\mathcal{Q}_{n}(x),i)|
\end{equation*}
From \cref{compact-space_convergence} second term goes to $0$ as \(n\to\infty\). Now
\begin{align*}
    |\bar\cJ_{k}(x,i)-\bar\cJ_{k}(\mathcal{Q}_{n}(x),i)|
   &\le \max_{\zeta\in \Act} \Bigg[ c_n(x,i,\zeta)-c_n(\mathcal{Q}_{n}(x),i,\zeta)+\int_{\cK\times\mathbb{S}} \bar\cJ_{k+1}(y,j)p_{\cK}(dy,dj\mid x,i,\zeta)\\
    &\qquad-\int_{\cK\times\mathbb{S}} \bar\cJ_{k+1}(y,j)p_{\cK}(dy,dj\mid \mathcal{Q}_{n}(x),i,\zeta)\Bigg]
\end{align*}
let 
\begin{align*}
&\Delta_{\mathcal{Q}_{n}}c=c_n(x,i,\zeta)-c_n(\mathcal{Q}_{n}(x),i,\zeta),\\ 
&\Delta_{\mathcal{Q}_{n}}\bar{\cJ}_{k+1}=\int_{\cK\times\mathbb{S}} \bar{\cJ}_{k+1}(y,j)p_{\cK}(dy,dj\mid x,i,\zeta)-\int_{\cK\times\mathbb{S}} \bar{\cJ}_{k+1}(y,j)p_{\cK}(dy,dj\;|\mathcal{Q}_{n}(x),i,\zeta).
\end{align*}
Since the running cost \(c\) is continuous, by the weak continuity of $p_{\cK}$ and the continuity and boundedness of \(\bar{\cJ}_{k+1}\) (see \cref{bdd&cts}), it follows that the terms \(\Delta_{\mathcal{Q}_{n}}c\) and \(\Delta_{\mathcal{Q}_{n}}\bar{\cJ}_{k+1}\to 0\) as \(n\to \infty\) uniformly on compact sets, hence \(\sup_{(x,i)\in \cK\times\mathbb S}|\bar{\cJ}_{k}(x,i)-\bar{\cJ}_{k}(\mathcal{Q}_{n}(x),i)|\to 0\;\; \text as\; n\to \infty\;\; \forall \;\; k=K-1,\ldots,0.\)
\end{proof}
\medskip
\noindent
Next, we show asymptotic optimality of the optimal controls designed from the models \(\widehat{\mathcal{M}}_n^{\cK}\) in the compact state model  \(\mathcal{M}^\cK\).

Let $\hat{v}_n^{*}$ be an optimal Markov policy for the finite-state model $\widehat{\mathcal{M}}_n^{\cK}$, obtained via the dynamic programming equation and justified by the verification theorem (\cite{hernandez2012discrete}*{Theorem 3.2.1}). Define a policy for the compact state model $\mathcal{M}^\cK$ by extending it to $\cK\times\bS$ (denoted by same $\hat{v}_n^{*}$)
\[
\hat{v}_n^{*}(k,x,i) \df \hat{v}_n^{*}(k,\mathcal{Q}_{n}(x),i)
\quad \forall\, x\in\cK\,, k\in\{0,\ldots,K\},\,i\in\mathbb{S}\]

The next result shows that the \(\hat{v}_n^{*}\) is asymptotically optimal in the compact space model \(\mathcal{M}_h\).
\begin{theorem}\label{finite-state-op-policy}
  Suppose that the assumptions of \cref{finite-state-cost_approximation} hold.  Let \(\hat{v}_n^{*}\in \Um\) be an optimal policy and \(\hat{\cJ}^{n*}_{k}\) be optimal value from time $k\in\{0,\ldots,K\}$ for the finite state model \(\widehat{ \mathcal{M}}_{n}^\cK\). Then for each $k\in\{0,\ldots,K\}$ 
    \begin{equation*}
       \lim_{n\to \infty} \sup_{(x,i)\in \cK\times\mathbb{S}}|\bar{\cJ}_{k}^{\hat{v}_n^{*}}(x,i)-\bar{\cJ}_{k}^{*}(x,i)| = 0
    \end{equation*}
\end{theorem}
\begin{proof}
    Consider the function \[\bar{\cJ}_{k}^{\hat{v}_n^{*}}(x,i)=c_n(x,i,\hat{v}_n^{*}(k,\mathcal{Q}_{n}(x),i))+\int_{\cK\times \mathbb{S}}\bar{\cJ}_{k+1}^{\hat{v}_n^{*}}(y,j)p_{\cK}(dy,dj\mid x,i,\hat{v}_n^{*}(k,\mathcal{Q}_{n}(x),i))\] Then
    \begin{equation*}
        |\bar{\cJ}_{k}^{\hat{v}_n^{*}}(x,i)-\bar{\cJ}_{k}^{*}(x,i)|\le
        |\bar{\cJ}_{k}^{\hat{v}_n^{*}}(x,i)-\hat \cJ_{k}^{n*}(x,i)|+
        |\hat \cJ_{k}^{n*}(x,i)-\bar{\cJ}_{k}^{*}(x,i)|
    \end{equation*}
From \cref{finite-state-cost_approximation}, second term \(|\hat \cJ_{k}^{n*}(x,i)-\bar{\cJ}_{k}^{*}(x,i)|\to 0\). Now
\begin{equation*}
    |\bar{\cJ}_{k}^{\hat{v}_n^{*}}(x,i)-\hat \cJ_{k}^{n*}(x,i)|\le  |\bar{\cJ}_{k}^{\hat{v}_n^{*}}(x,i)-\bar{\cJ}_{k}^{\hat{v}_n^{*}}(\mathcal{Q}_{n}(x),i)|+|\bar{\cJ}_{k}^{\hat{v}_n^{*}}(\mathcal{Q}_{n}(x),i)-\hat \cJ_{k}^{n*}(x,i)|
    \end{equation*}
    Arguing as in the proof of \cref{compact-space_convergence} we get \(|\bar{\cJ}_{k}^{\hat{v}_n^{*}}(\mathcal{Q}_{n}(x),i)-\hat \cJ_{k}^{n*}(x,i)|\to 0\) as \(n\to \infty\). Since
    \begin{align*}
   |\bar{\cJ}_{k}^{\hat{v}_n^{*}}(x,i)-\bar{\cJ}_{k}^{\hat{v}_n^{*}}(\mathcal{Q}_{n}(x),i)| &\le \bigg|c_n(x,i,\hat{v}_n^{*}(k,\mathcal{Q}_{n}(x),i))-c_n(\mathcal{Q}_{n}(x),i,\hat{v}_n^{*}(k,\mathcal{Q}_{n}(x),i))\\
   &+\int_{\cK\times \mathbb{S}}\bar{\cJ}_{k+1}^{\hat{v}_n^{*}}(y,j)p_{\cK}(dy,dj\mid x,i,\hat{v}_n^{*}(k,\mathcal{Q}_{n}(x),i))\\&\quad-\int_{\cK\times \mathbb{S}} \bar{\cJ}_{k+1}^{\hat{v}_n^{*}}(y,j)p_{\cK}(dy,dj|\mathcal{Q}_{n}(x),i,\hat{v}_n^{*}(k,\mathcal{Q}_{n}(x),i))\bigg|
\end{align*}
By the Lipschitz continuity of running cost \(c_n\), the weak continuity of $p_{\cK}$ and the boundedness of \(\bar{\cJ}_{k+1}\), we obtain \(|\bar{\cJ}_{k}^{\hat{v}_n^{*}}(x,i)-\bar{\cJ}_{k}^{\hat{v}_n^{*}}(\mathcal{Q}_{n}(x),i)|\to 0\) as \(n\to\infty.\) Hence, combining the above estimates, the result follows.   
\end{proof}

\medskip
\noindent
We now show that the original MDP $\mathcal{M}_h$ can be approximated by
finite-state models $\widehat{\mathcal{M}}_{h,n}$. To this end, we construct
a sequence of compact-state MDPs and use the results
of the previous sub-subsection, which establishes the approximation of a compact-state
MDP by finite-state MDPs.

This yields the approximation scheme
\[
\widehat{\mathcal{M}}_{h,n}
\;\longrightarrow\;
\mathcal{M}_h^{\cX_n}
\;\longrightarrow\;
\mathcal{M}_h.
\]
\subsubsection{Finite State Approximation of $\mathcal{M}_h$}
Let $\cK_n$ be a sequence of compact sets in $\Rd$ such that $\cK_n \subset int\,\cK_{n+1}$ and $\Rd=\bigcup_{n\ge 1} \cK_n$. Let $\{\mu_n\}_{n\ge 1}$ be a sequence of probability measures such that for each $n\geq1$, $\mu_n \in \sP(\cK_n^c)$. Similar to finite state MDPs \(\widehat{\mathcal{M}}_{n}^{\cK}\) construction,
we construct a sequence of compact state MDPs denoted as $\mathcal{M}_h^{\cX_n}$ to approximate the orignal MDP $\mathcal{M}_h$. To this end, for each $n$ let $\cX_n=\cK_n\cup\{\Delta_n\}$   where $\Delta_n\in\cK_n^c$ is a pseudo state. We define the transition probability \(p_{\cX_n}:\cX_n\times\mathbb{S}\times\Act\to\sP(\cX_n\times\mathbb{S})\), one-stage cost function \(c_{\cX_n}:\cX_n\times\mathbb{S}\times\Act\to [0,\infty)\) and the terminal cost by 

\begin{align*}
p_{\cX_n}(\cdot,j\mid x,i,\zeta)
&=\begin{cases} P_h\big(\cdot \cap\, \cK_n,\,j \mid x,i,\zeta\big)+P_h\big( \cK_n^c,\,j \mid z,i,\zeta\big)\delta_{\Delta_n}, \,\,&\text{if}\,\,(x,i)\in\cK_n\times\mathbb{S},\\
\int_{\cK_n^c} \Big( P_h\bigl(\cdot \cap \cK_n,j \mid z,i,\zeta\bigr)+ P_h\bigl(\cK_n^c,j\mid z,i,\zeta\bigr) \delta_{\Delta_n} \Big) \mu_n(dz), \,\,&\text{if}\,\,(x,i)\in\{\Delta_n\}\times\mathbb{S},
\end{cases}\\
c_{\cX_n}(x,i,\zeta)&=
\begin{cases}
 c_h(x,i,\zeta),   &\text{ if } (x,i)\in \cK_n\times\mathbb{S}  \\
\int_{ \cK_n^c} c_h(z,i,\zeta) \mu_n(dz) ,  &\text{ if } x=\Delta_n,\;i\in\mathbb{S}. \nonumber
\end{cases}\\
\hat c_T^n(x,i)&=
\begin{cases}
  c_T(x,i),   &\text{ if } (x,i)\in \cK_n\times\mathbb{S}  \\
\int_{ \cK_n^c} c_T(z,i) \mu_n(dz) ,  &\text{ if } x=\Delta_n,\;i\in\mathbb{S}. 
\end{cases}
\end{align*}
So we have a sequence of compact-state MDPs $\mathcal{M}_{h}^{\cX_n}=(\cX_n\times\mathbb{S},\Act,p_{\cX_n},c_{\cX_n})$

To establish the main result of this section, we introduce, for each $n$, another MDP, denoted by $\mathcal{M}^{\Rd}_n$, with the components $\bigl(\Rd\times\mathbb{S},\Act, q_n,r_n)$ where
\begin{align}
q_n(\,\cdot,j\,|x,i,\zeta) &= \begin{cases}
P_h(\,\cdot\,,j\mid x,i,\zeta),   &\text{ if } x\in \cK_n \,,i\in\mathbb{S} \\
\int_{\cK_n^c} P_h\bigl(\,\cdot\,,j\mid z,i,\zeta) \mu_n(dz) ,  &\text{ if } x \in \cK_n^c\,,i\in\mathbb{S},
\end{cases} \nonumber \\
r_n(x,i,\zeta) &= \begin{cases}
c_h(x,i,\zeta),   &\text{ if } x\in \cK_n \,,i\in\mathbb{S} \\
\int_{\cK_n^c} c_h(z,i,\zeta) \mu_n(dz) ,  &\text{ if } x \in \cK_n^c\,,i\in\mathbb{S}. \nonumber
\end{cases}
\end{align}
The terminal cost is \[\hat c_T^n=
\begin{cases}
c_T(x,i),   &\text{ if } x\in \cK_n \,,i\in\mathbb{S} \\
\int_{\cK_n^c} c_T(z,i) \mu_n(dz) ,  &\text{ if } x \in \cK_n^c\,,i\in\mathbb{S}
\end{cases}\]
For each policy $\pi \in \Um^h$ and initial distribution $\nu \in \sP(\Rd\times\mathbb{S})$, we denote the finite-horizon cost functions for $\mathcal{M}^{\Rd}_n$ by $\cJ_{k,h}^{n,\pi}(\nu)$ and $\cJ_{k,h}^{n*}(\pi,\nu)$.

\medskip
\noindent
Before approximating the cost functional, we first present several auxiliary results.
We begin with a standard weak continuity property of the transition kernel, which follows from the continuity of the coefficients of the controlled regime-switching diffusion and the construction of the Euler--Maruyama scheme; see, e.g., \cite{pradhan2025nearoptimalitydiscretetimeapproximations}*{Theorem 4.2}.

\begin{lemma}[Weak continuity of the transition kernel]\label{lem:weak_continuity}
Suppose that Assumption \hyperlink{B1}{(B1)} holds. Then the transition kernel $P_h(dy,j\mid x,i,\zeta)$ is weakly continuous in $(x,\zeta)$, i.e., for any bounded continuous function $f:\Rd\times\mathbb{S}\to\RR$,
\[
\int_{\Rd\times\mathbb{S}} f(y,j)\,P_h(dy,dj\mid x_n,i,\zeta_n)
\;\longrightarrow\;
\int_{\Rd\times\mathbb{S}} f(y,j)\,P_h(dy,dj\mid x,i,\zeta),
\]
whenever $(x_n,\zeta_n)\to(x,\zeta)$.
\end{lemma}

Next, we state a lemma, adapted from \cite{MR3722422}*{Lemma 3.1} (see also \cite{DUFOUR20121254}*{Lemma 2.9}), which ensures that if the initial state lies in a compact set, then with high probability the controlled process remains in a compact set over all time horizons.
\begin{lemma}\label{lem:compact_containment}
Suppose that the control model $\mathcal{M}_h$ satisfies assumption {\hyperlink{B1}{(B1)}}. For any compact subset $\cK$ of $\Rd$ and for any $\varepsilon>0$, there exists a compact subset $\cK_{\varepsilon}$ of $\Rd$ such that
\begin{align}
\sup_{(x,i,\zeta) \in \cK\times\mathbb{S}\times\Act}  P_h(\cK_{\varepsilon}^c,\mathbb{S}\mid x,i,\zeta) < \varepsilon,\nonumber
\end{align}
where $\cK_{\varepsilon}^c$ denotes the complement of the set $\cK_{\varepsilon}$.
\end{lemma}

The following lemma guarantees that the MDP $\mathcal{M}_{h}^{\cX_n}$ and $\mathcal{M}^{\Rd}_n$ are equivalent
\begin{lemma}\label{lemma1}
Under assumption \hyperlink{B2}{(B2)}, for each $k\in\{0,\ldots,K\}$  we have
\begin{align}
\cJ_{k,h}^{n*}(x,i) = \begin{cases}
\bar{\cJ}_{k,h}^{n*}(x,i),   &\text{ if } x\in \cK_n \,,i\in\mathbb{S} \\
\bar{\cJ}_{k,h}^{n*}(\Delta_n,i) ,  &\text{ if } x \in \cK_n^c\,,i\in\mathbb{S},
\end{cases}\label{eq50}
\end{align}
where $\cJ_{k,h}^{n*}$ is the finite-horizon value function of $\mathcal{M}^{\Rd}_n$ and $\bar{\cJ}_{k,h}^{n*}$ is the finite-horizon value function of $\mathcal{M}_{h}^{\cX_n}$, provided that there exist optimal deterministic Markov policies for $\mathcal{M}^{\Rd}_n$ and $\mathcal{M}_{h}^{\cX_n}$. Furthermore, if for any deterministic Markov policy $\bar v$, we define ${v}(k,x,i)=\bar v(k,x,i)$ on $\cK_n\times\bS$ and $v(k,x,i)=\bar v(k,\Delta_n,i)$ on $\cK_n^c\times\bS$, then
\begin{align}
\cJ_{k,h}^{n,v}(x,i) = \begin{cases}
\bar{\cJ}_{k,h}^{n,\bar v}(x,i),   &\text{ if } x\in \cK_n \,,i\in\mathbb{S} \\
\bar{\cJ}_{k,h}^{n,\bar v}(\Delta_n,i) ,  &\text{ if } x \in \cK_n^c\,,i\in\mathbb{S}.
\end{cases}\label{eq51}
\end{align}
In particular, if the deterministic Markov policy $\bar v_n^* $ is optimal for $\mathcal{M}_{h}^{\cX_n}$, then its extension ${v}_n^*$ to $\Rd$ is also optimal for $\mathcal{M}^{\Rd}_n$.
\end{lemma}

\begin{proof}
    See the Appendix.
\end{proof}

The following result gives us the convergence of value function of $\mathcal{M}^{\Rd}_n$ to the value function of the orignal MDP $\mathcal{M}_{h}$.

\begin{lemma}\label{lemma2}
Under assumptions \hyperlink{B1}{(B1)}-\hyperlink{B2}{(B2)}. For any compact set $\cK\subset \Rd$, $k\in\{0,\ldots,K\}$, we have
\begin{align}
\lim_{n\rightarrow\infty} \sup_{(x,i)\in \cK\times\mathbb{S}} |\cJ_{k,h}^{n*}(x,i) - \cJ_{k,h}^*(x,i)| = 0 \label{eq17}
\end{align}
\end{lemma}

\begin{proof}

We prove \eqref{eq17} by backward induction on $k$.

For $k=K$,
let $\cK \subset \Rd$ be compact. For sufficiently large $n$, we have
$\cK \subset \cK_n$, and hence $r_n = c_h$ on $\cK$.
Therefore,
\(
\cJ_{K,h}^{n*}(x,i) = \cJ_{K,h}^*(x,i) = \hat c_T^n(x,i),
\quad \forall\, (x,i)\in \cK \times \mathbb{S}.
\)
Now assume the claim holds for some $k \in\{0,\ldots,K-1\}$, and fix a compact set
$\cK \subset \Rd$. Recall the compact set $\cK_\varepsilon$ from
\cref{lem:compact_containment}. By the construction of $q_n$ and $r_n$, there exists
$n_0 \ge 1$ such that for all $n \ge n_0$,
\(
q_n = P_h, \quad r_n = c_h \quad \text{on } \cK.
\)
With these observations, for each $n \ge n_0$, we have
\begin{align}
\sup_{(x,i)\in\cK\times\mathbb{S}} |{\cJ}_{k,h}^{n*}(x,i) -\cJ_{k,h}^*(x,i)|
&= \sup_{(x,i)\in \cK\times\mathbb{S}} \biggl| \inf_{\zeta\in\Act} \biggl[ c_h(x,i,\zeta) +  \hspace{-5pt} \int_{\Rd\times\mathbb{S}} \cJ_{k+1,h}^{n*}(y,j) P_h(dy,dj\mid x,i,\zeta) \biggr]  \nonumber \\
&\qquad- \inf_{\zeta\in\Act} \biggl[ c_h(x,i,\zeta) +  \hspace{-5pt} \int_{\Rd\times\mathbb{S}} \cJ_{k+1,h}^*(y,j) P_h(dy,dj\mid x,i,\zeta) \biggr] \biggr|  \nonumber \\
&\leq  \sup_{(x,i,\zeta) \in \cK\times\Act} \biggl| \int_{\Rd\times\mathbb{S}} \cJ_{k+1,h}^{n*}(y,j) P_h(dy,dj\mid x,i,\zeta)  \nonumber \\
&\qquad- \int_{\Rd\times\mathbb{S}} \cJ_{k+1,h}^*(y,j) P_h(dy,dj\mid x,i,\zeta) \biggr| \nonumber \\
&=  \sup_{(x,i,\zeta) \in \cK\times\Act} \biggl| \int_{\cK_{\varepsilon}} \bigl(\cJ_{k+1,h}^{n*}(y,j) - \cJ_{k+1,h}^*(y,j)\bigr) \hspace{3pt} P_h(dy,dj\mid x,i,\zeta)  \nonumber \\
&\qquad+ \int_{\cK_{\varepsilon}^c} \bigl(\cJ_{k+1,h}^{n*}(y,j) - \cJ_{k+1,h}^*(y,j)\bigr) \hspace{3pt} P_h(dy,dj\mid x,i,\zeta)\biggr| \nonumber\\
&\leq  \sup_{(x,i)\in \cK_{\varepsilon}\times\mathbb{S}} |\cJ_{k+1,h}^{n*}(x,i) - \cJ_{k+1,h}^*(x,i)|  \nonumber \\
&\qquad+ \sup_{(x,i,\zeta) \in \cK\times\Act} \biggl| \int_{\cK_{\varepsilon}^c} \bigl(\cJ_{k+1,h}^{n*}(y,j) - \cJ_{k+1,h}^*(y,j)\bigr) \hspace{3pt} P_h(dy,dj\mid x,i,\zeta)\biggr| \nonumber
\end{align}
Note that we have $|{\cJ}_{k+1,h}^{n*}|,\,|\cJ_{k+1,h}^*| \leq M_3$ by \hyperlink{B2}{(B2)}. 
Then by \cref{lem:compact_containment} we have
\begin{align}
\sup_{(x,i)\in\cK\times\mathbb{S}} |\cJ_{k,h}^{n*}(x,i) -\cJ_{k,h}^*(x,i)| &\leq  \sup_{(x,i)\in \cK_{\varepsilon}\times\mathbb{S}} |\cJ_{k+1,h}^{n*}(x,i) - \cJ_{k+1,h}^*(x,i)| +   2M_3 \epsilon. \nonumber
\end{align}
Since the first term converges to zero as $n\rightarrow\infty$ by the induction hypothesis, and $\epsilon$ is arbitrary, the claim is true for $k$. This completes the proof. 
\end{proof}
\medskip
\noindent
Now we compute a near-optimal policy for the original MDP $\mathcal{M}_h$ using \cref{finite-state-op-policy} and the above results. Note that continuity of $c_{\cX_n}$ and weak continuity of $p_{\cX_n}$ follow from the continuity of $c_h$ and the weak continuity of $P_h$, respectively. Hence, for each $n$, it is easy to check the MDP $\mathcal{M}_h^{\cX_n}$ satisfies the assumptions
of \cref{finite-state-op-policy}. Let $\{\varepsilon_n\}$ be a sequence of positive real numbers such that 
$\varepsilon_n \to 0$ as $n \to \infty$.

By \cref{finite-state-op-policy}, for each $n\geq1$, there exists an optimal policy $v_n$, obtained from the finite state approximations of $\mathcal{M}_{h}^{\cX_n}$, such that 
\begin{align}
\sup_{(x,i) \in \cX_n\times\mathbb{S}}| \bar \cJ_{k,h}^{n,v_n}(x,i) - \bar{\cJ}_{k,h}^{n*}(x,i) | \leq \varepsilon_n, \nonumber 
\end{align}
where for each $n$, finite-state models $\widehat{\mathcal{M}}_{h,n}$ are constructed replacing $\bigl( \cK\times\mathbb{S},\Act,p_{\cK},c_{\cK} \bigr)$ with the components $\bigl( \cX_n\times\mathbb{S},\Act,p_{\cX_n},c_{\cX_n}\bigr)$ of $\mathcal{M}_{h}^{\cX_n}$ in the previous sub-subsection. By \cref{lemma1}, for each $n\geq1$ we also have
\begin{align}\label{eq22}
\sup_{(x,i) \in \Rd\times\mathbb{S}}| \cJ_{k,h}^{n,v_n}(x,i) - \cJ_{k,h}^{n*}(x,i) | \leq \varepsilon_n, 
\end{align}
where, with an abuse of notation, we also denote the extended (to $\Rd\times\mathbb{S}$) policy by $v_n$. Let us define the finite-horizon cost corresponding to the policy $v_n$ for MDP $\mathcal{M}_{n}^{\Rd}$  as
\begin{align*}
\cJ_{k,h}^{n,v_n}(x,i)  = \begin{cases}
c_h(x,i,v_n(k,x,i)) \\
\qquad
+  \int_{\Rd\times\mathbb{S}} {\cJ}_{k+1,h}^{n,v_n} (y,j) P_h(dy,dj\mid x,i,v_n(k,x,i)),  &\text{ if } (x,i)\in \cK_n\times\mathbb{S}  \\
\int_{\cK_n^c} \bigl[ c_h(z,i,v_n(k,z,i)) \\
\qquad+  \int_{\Rd\times\mathbb{S}} {\cJ}_{k+1,h}^{n,v_n}(y,j) P_h(dy,dj\mid z,i,v_n(k,z,i)) \bigr] \mu_n(dz),  &\text{ if }(x,i) \in \cK_n^c\times\mathbb{S},
\end{cases} 
\end{align*}
For each n, \cite{hernandez2012discrete}*{Theorem 3.2.1} implies that there exists an optimal Markov policy $v_n^*$ for MDP $\mathcal{M}_{n}^{\Rd}$
which satisfies the optimality equation. Hence, we have
\[\cJ_{k,h}^{n*}(x,i)=\cJ_{k,h}^{n}(v_n^*,x,i)\]
Similarly, for the  original MDP $\mathcal{M}_{h}$, the finite-horizon cost corresponding to $v_n$ \text{as}
\begin{equation*}
\cJ_{k,h}^{v_n}(x,i)= c(x,i,v_n(k,x,i)) +  \int_{\Rd\times\mathbb{S}} \cJ_{k+1,h}^{v_n}(y,j) P_h(dy,dj\mid x,i,v_n(k,x,i)).
\end{equation*}

\begin{lemma}\label{lemma3} Suppose assumptions \hyperlink{B1}{(B1)}-\hyperlink{B2}{(B2)} holds. Then, 
for any compact set $\cK \subset \Rd$, we have
\begin{align}
\lim_{n\rightarrow\infty} \sup_{(x,i) \in \cK\times\mathbb{S}} |\cJ_{k,h}^{n,v_n}(x,i) - \cJ_{k,h}^{v_n}(x,i)| = 0. \nonumber 
\end{align}
Indeed, this is true for all sequences of policies in $\Um^h$.
\end{lemma}

\begin{proof}
The lemma can be proved using arguments similar to those in the proof of \cref{lemma2}, and we omit the details. 
\end{proof}

The following theorem is the main result of this subsection and states that the true cost functions of policies obtained from finite-state models converge to the value function of the original MDP. Hence, to obtain a near-optimal policy for the original MDP, it is sufficient to compute the optimal policy for the finite state model that has sufficiently large number of grid points.

\begin{theorem}\label{mainthm2} Under assumptions \hyperlink{B1}{(B1)}-\hyperlink{B2}{(B2)}.
For any compact set $\cK \subset \Rd$, we have
\begin{align}
\lim_{n\rightarrow\infty} \sup_{(x,i)\in \cK\times\mathbb{S}} |\cJ_{k,h}^{v_n}(x,i) - \cJ_{k,h}^*(x,i)| &= 0. \nonumber 
\intertext{Therefore,}
\lim_{n\rightarrow\infty} |\cJ_{k,h}^{v_n}(x,i) - \cJ_{k,h}^*(x,i)| &= 0 \text{  } \text{ for all $(x,i) \in \Rd\times\mathbb{S}$}. \nonumber 
\end{align}
\end{theorem}

\begin{proof}
The result follows from \cref{lemma2}, \cref{eq22}, and \cref{lemma3}. 
\end{proof}
\medskip
\noindent
Now we state our main result of this section: the optimal control designed for the finite model is asymptotically optimal in the controlled RSDP.

\begin{theorem}
  Suppose that Assumptions \hyperlink{B1}{(B1)}--\hyperlink{B2}{(B2)} hold. Let \(\hat {v}_n^{*,h}\) be the extended (to $\Rd$) optimal policy obtained from the finite state models \(\widehat {\mathcal{M}}_{h,n}\). Then for any compact set \(\cK\subset \Rd\), 
    \begin{equation*}
       \lim_{h\to 0}\lim_{n\to\infty} \sup_{(x,i)\in \cK\times\mathbb S}|{\cJ}_T^{\hat {v}_n^{*,h}}(x,i)-{\cJ}_T^*(x,i)|= 0 \qquad 
    \end{equation*}
    Therefore,
\begin{equation*}
\lim_{h\to 0}\lim_{n\rightarrow\infty} |\cJ_{T}^{\hat {v}_n^{*,h}}(x,i) - \cJ_{T}^*(x,i)| = 0\quad  \forall\, (x,i) \in \Rd\times\mathbb{S}
\end{equation*}
\end{theorem}
\begin{proof} By the triangle inequality
   \begin{align*}
       |{\cJ}_T^{\hat {v}_n^{*,h}}(x,i)-{\cJ}_T^*(x,i)|&\le |{\cJ}_T^{\hat {v}_n^{*,h}}(x,i)-{\cJ}_{T,h}^{\hat {v}_n^{*,h}}(x,i)|+|{\cJ}_{T,h}^{\hat {v}_n^{*,h}}(x,i)- \cJ_{T,h}^{*}(x,i)|\\
       &\quad+|{\cJ}_{T,h}^{*}(x,i)-{\cJ}_T^*(x,i)|
   \end{align*} 
   From \cref{cor:CNearFinite,mainthm2,Tconvevaluefunc} we obtain \(\sup_{(x,i)\in \cK\times\mathbb S}|{\cJ}_T^{\hat {v}_n^{*,h}}(x,i)-{\cJ}_T^*(x,i)|\to 0\) as \(n\to\infty\) and \(h\to 0\).
\end{proof}

\section{Conclusion}
In this paper, we studied approximation of optimal control for state-dependent controlled regime-switching diffusions. We proved continuity of cost functionals under the Borkar topology and, using known density results, established near-optimality of finite-action, piecewise-constant, and Lipschitz policies.

We developed an approximation framework combining structural simplification of controls with time and state discretization for the finite-horizon problem. In particular, we constructed an Euler--Maruyama scheme under piecewise-constant controls and showed convergence of both the state process and value functions. We also introduced a finite-state approximation via quantization and proved uniform convergence on compact sets and asymptotic optimality of the resulting policies.

Our results assume non-degenerate dynamics; extending them to degenerate models remains an open problem. Other directions for future work include convergence results for general admissible control policies, finite-state approximation for infinite-horizon cost criteria, extensions to partially observed systems, sharper convergence rates, and the development of efficient computational methods.
\section*{Acknowledgment}
This research of the first author was partially supported by a Start-up Grant IISERB/ R\&D/2024-25/154 and Prime Minister Early Career Research Grant ANRF/ECRG/2024/001658/ PMS. The research of the second author was partially supported by the UGC Junior Research Fellowship (UGC-JRF), Ministry of Education, Government of India.


\appendix
\section{}
The following Lemma is adapted from \cite{nguyen2025hybrid}*{Lemma 2.5}
\begin{lemma}\label{t-3.1}
For any bounded and measurable function $f:\mathbb{S}\times\mathbb{S}\to\RR$ and 
$(x,i,\bar{x},j)\in\Rd\times\mathbb{S}\times\Rd\times\mathbb{S}$, we have
\begin{align}\label{2.30}
\int_{\RR_+} \bigl[f(i+h(x,i,z),\, i+h(\bar x,i,z)) - f(i,i)\bigr]\,\mathbf m(dz)
&= \sum_{k\neq i} (m_{ik}(x)-m_{ik}(\bar x))^+ \bigl(f(k,i)-f(i,i)\bigr) \nonumber \\
&\quad + \sum_{k\neq i} (m_{ik}(\bar x)-m_{ik}(x))^+ \bigl(f(i,k)-f(i,i)\bigr) \nonumber \\
&\quad + \sum_{k\neq i} (m_{ik}(x)\wedge m_{ik}(\bar x)) \bigl(f(k,k)-f(i,i)\bigr),
\end{align}
and for $i\neq j$,
\begin{align}\label{2.31}
\int_{\RR_+} \bigl[f(i+h(x,i,z),\, j+h(\bar{x},j,z)) - f(i,j)\bigr]\,\mathbf m(dz)
&= \sum_{k\neq i} m_{ik}(x)\bigl(f(k,j)-f(i,j)\bigr) \nonumber \\
&\quad + \sum_{l\neq j} m_{jl}(\bar x)\bigl(f(i,l)-f(i,j)\bigr).
\end{align}
Consequently, for any $i\in\mathbb{S}$ and $x,\bar{x}\in\Rd$, we have
\begin{equation}\label{2.32}
\int_{\RR_+} \mathbf 1_{\{h(x,i,z)\neq h(\bar x,i,z)\}}\,\mathbf m(dz)
\le \sum_{k\neq i} |m_{ik}(x)-m_{ik}(\bar x)|.
\end{equation}
\end{lemma}

\medskip

\begin{proof}[Proof of \cref{lemma1}]
The identity \cref{eq51} for a fixed policy follows from the fact that for any $x,y \in \cK_n^c$, $i,j \in \bS$ and $\zeta \in \Act$,
\[
r_n(x,i,\zeta) = r_n(y,i,\zeta), 
\qquad 
q_n(\cdot,j \mid x,i,\zeta) = q_n(\cdot,j \mid y,i,\zeta).
\]
Hence, under the extended policy $v$, all states $(x,i)$ with $x \in \cK_n^c$ have identical one-stage costs and transition probabilities, and therefore behave identically to the pseudo-state $(\Delta_n,i)$. This yields \cref{eq51}.

\medskip

To prove \cref{eq50}, it suffices to show that
\begin{equation}\label{eq}
\inf_{v \in {\Uadm}_{\mathsf{m}}^h} \cJ_{k,h}^{n,v}(x,i)
=
\inf_{v \in \bar{\Uadm}_{\mathsf{m}}^h} \cJ_{k,h}^{n,v}(x,i),
\end{equation}
where $\bar{\Uadm}_{\mathsf{m}}^h \subset {\Uadm}_{\mathsf{m}}^h$ denotes the class of policies that are constant in $x$ over $\cK_n^c$ for each fixed $i$.

The inequality
\begin{equation}\label{ineq}
\inf_{v \in {\Uadm}_{\mathsf{m}}^h} \cJ_{k,h}^{n,v}(x,i)
\le
\inf_{v \in \bar{\Uadm}_{\mathsf{m}}^h} \cJ_{k,h}^{n,v}(x,i)
\end{equation}
is immediate since $\bar{\Uadm}_{\mathsf{m}}^h \subset {\Uadm}_{\mathsf{m}}^h$.

\medskip

We prove by backward induction that if $v \in {\Uadm}_{\mathsf{m}}^h$ is an optimal policy then there exists an optimal policy $\hat v \in \bar{\Uadm}_{\mathsf{m}}^h$ satisfying
\begin{align}\label{eq2}
 \cJ_{k',h}^{n,v}(x,i)
= \cJ_{k',h}^{n,\hat v}(x,i)\quad \forall\,k'\ge k
\end{align}
This immediately implies \cref{eq}. For $k'=K$, \cref{eq2} trivially holds since $\cJ_{K,h}^{n,\hat v}(x,i)=\cJ_{K,h}^{n,v}(x,i)=\hat c_T(x,i)$ for all $(x,i)\in\Rd\times\bS$. Assume that there exists a policy
$\hat v^{\,k+1}\in\bar{\Uadm}_{\mathsf m}^h$
such that
\[
\cJ_{k+1,h}^{n,v}(x,i)
=
\cJ_{k+1,h}^{n,\hat v^{\,k+1}}(x,i),
\qquad
\forall (x,i)\in\Rd\times\bS.
\] For the induction step at stage $k$, define a new Markov policy
$\hat v^{\,k}$ by modifying $v$ only at stages $k,k+1,\ldots,K-1$.
Fix $z \in \cK_n^c$ and define  $\hat v^{\,k}$ by
\begin{align*}
\hat v^{\,k}(k,x,i)&=
\begin{cases}
v(k,x,i), & (x,i)\in\cK_n\times\bS,\\
v(k,z,i), & (x,i)\in\cK_n^c\times\bS,
\end{cases}\\
\text{and}\qquad
\hat v^{\,k}(k',x,i)
&=
\hat v^{\,k+1}(k',x,i),
\qquad k'>k.
\end{align*}
\medskip
\noindent
\textbf{Claim:} $ \cJ_{k,h}^{n,v}(x,i)
= \cJ_{k,h}^{n,\hat v^{\,k}}(x,i)
$.

\medskip
\noindent
To prove the induction step at stage $k$, using the Bellman equation for $(x,i)\in\cK_n^c\times\bS$, we have
\begin{align}
    \cJ_{k,h}^{n,\hat v^{\,k}}(x,i)
&=
r_n(x,i,\hat v^{\,k}(k,x,i))
+
\int_{\Rd\times\bS}
\cJ_{k+1,h}^{n,\hat v^{\,k}}(y,j)\,
q_n(dy,dj \mid x,i,\hat v^{\,k}(k,x,i)),\label{bellman}\\
&=
r_n(z,i, v(k,z,i))
+
\int_{\Rd\times\bS}
\cJ_{k+1,h}^{n,v}(y,j)\,
q_n(dy,dj \mid z,i, v(k,z,i)),\nonumber\\
&=\cJ_{k,h}^{n,v}(z,i)\nonumber
\end{align}
where the second equality follows from the construction of
$r_n$ and $q_n$, together with the induction hypothesis
$\cJ_{k+1,h}^{n,\hat v^{\,k}}
=
\cJ_{k+1,h}^{n,v}$.

\medskip
\noindent
Now, we have two cases

\medskip

\noindent
\textbf{Case 1:} $\cJ_{k,h}^{n,v}(x,i)$ is constant on $\cK_n^c$.

\medskip
\noindent
Since 
$\hat v^{\,k}$ and $v$ coincide on $\cK_n$, and $\cJ_{k,h}^{n,v}$ is constant on $\cK_n^c$
, from \cref{bellman}, we have
\[
\cJ_{k,h}^{n,\hat v^{\,k}}(x,i)
=
\cJ_{k,h}^{n,v}(x,i),
\quad \forall\, (x,i)\in \Rd\times\bS.
\]

\medskip

\noindent
\textbf{Case 2:} $\cJ_{k,h}^{n,v}(x,i)$ is not constant on $\cK_n^c$.

Then there exist $z,y \in \cK_n^c$ such that
\[
\cJ_{k,h}^{n,v}(z,i) < \cJ_{k,h}^{n,v}(y,i),\qquad i\in\bS
\]
Define $\hat v^{\,k}$ as above using $z$.
Since $r_n$ and $q_n$ do not depend on $(x,i) \in \cK_n^c\times\bS$, 
using the Bellman equation \cref{bellman},
we obtain
\begin{align}\label{bellman2}
\cJ_{k,h}^{n,\hat v^{\,k}}(y,i)
=\cJ_{k,h}^{n,v}(z,i)<\cJ_{k,h}^{n,v}(y,i)
\end{align}
since $ \cJ_{k,h}^{n,v}(y,i)\le \cJ_{k,h}^{n,\hat v^{\,k}}(y,i)$ from \cref{bellman2}, we obtain
\[ \cJ_{k,h}^{n,v}(y,i)
<
 \cJ_{k,h}^{n, v}(y,i)\]
 which is a contradiction. Hence, Case 2 is impossible.
Therefore
\(
\cJ_{k,h}^{n,v}
\)
is constant on
\(
\cK_n^c\times\bS.
\). Thus, we have
\[
\cJ_{k,h}^{n,\hat v^{\,k}}(x,i)
=
\cJ_{k,h}^{n,v}(x,i),
\qquad
\forall (x,i)\in\Rd\times\bS.
\]
This proves the induction step. Therefore, \eqref{eq2} holds for every stage
\(k'=K,K-1,\ldots,k\),
and consequently \eqref{eq} follows. Since \(v\) is optimal, \(\hat v^{\,k}\) attains the optimal value function at every state, and is therefore an optimal policy. Repeating the same argument for
\(k=K-1,K-2,\ldots,0\),
one obtains an optimal policy that is constant on
\(\cK_n^c\) at every stage.
Finally, if $\bar v_n^*$ is optimal for $\mathcal{M}_h^{\cX_n}$, then its extension $v_n^*$ attains the optimal value in $\mathcal{M}_h^{\Rd}$, and is therefore optimal.
\end{proof}


\bibliography{Dinesh_Robustness}
\bibliographystyle{elsarticle-num} 

\end{document}